\documentclass[a4paper,11 pt,reqno]{amsart}
\usepackage{a4,amsbsy,amscd,amssymb,amstext,amsmath,latexsym,oldgerm,enumerate,verbatim,graphicx,xy}

\makeatletter                                       
\def\l@subsection{\@tocline{2}{0pt}{2.5pc}{2.5pc}{}}
\makeatother                                        

\makeatletter                                                  
\def\chapter{\clearpage\thispagestyle{plain}\global\@topnum\z@ 
\@afterindenttrue \secdef\@chapter\@schapter}
\makeatother

\newtheorem{thm} {Theorem} [section]
\newtheorem{prop}{Proposition} [section]
\newtheorem{lem} {Lemma} [section]

\newtheorem{cornn}{Corollary}

\theoremstyle{definition}

\newtheorem{rem} {Remark} [section]
\newtheorem{rems} [rem]{Remarks}

\newcommand{\mf}{\mathfrak}
\newcommand{\mc}{\mathcal}
\newcommand{\mb}{\mathbb}

\newcommand{\ov}{\overline}

\newcommand{\sm}{\setminus}         

\newcommand{\ot}{\otimes}           
\newcommand{\la}{\langle}
\newcommand{\ra}{\rangle}

\newcommand{\Hom}{{\rm Hom}}        
\newcommand{\End}{{\rm End}}

\newcommand{\Mat}{{\rm Mat}}

\newcommand{\Spec}{{\rm Spec}}
\newcommand{\Maxspec}{{\rm Maxspec}}

\newcommand{\Ker}{{\rm Ker}}

\newcommand{\rk}{{\rm rk}}

\let\ttie\t
\newcommand{\tie}[1]{{\let\t\ttie \ttie#1}}
\renewcommand{\t}{\mf{t}}  

\newcommand{\GL}{{\rm GL}}

\newcommand{\ve}{\varepsilon}

\xyoption{all}

\begin{document}

\title{On embeddings of certain spherical homogeneous spaces in prime characteristic}

\begin{abstract}
Let $\mc G$ be a reductive group over an algebraically closed field of characteristic $p>0$.
We study embeddings of homogeneous $\mc G$-spaces that are induced from the $G\times G$-space $G$, $G$ a suitable reductive group,
along a parabolic subgroup of $\mc G$. We give explicit formulas for the canonical divisors and for the divisors of
$B$-semi-invariant functions. Furthermore, we show that, under certain mild assumptions, any (normal) equivariant
embedding of such a homogeneous space is canonically Frobenius split compatible with certain subvarieties and has an equivariant
rational resolution by a toroidal embedding. In particular, all these embeddings are Cohen-Macaulay. 

Examples are the $G\times G$-orbits in normal reductive monoids with unit group $G$.
Further examples are the open $\mc G$-orbits of the well-known determinantal varieties and the varieties
of (circular) complexes.

Finally we study the Gorenstein property for the varieties of circular complexes and for a related reductive monoid.
\end{abstract}

\author[R.\ H.\ Tange]{Rudolf Tange}


\maketitle{}
\markright{\MakeUppercase{On embeddings of certain spherical homogeneous spaces}}

\section*{Introduction}
Let $k$ is an algebraically closed field of prime characteristic $p$ and let $\mc G$ be a connected reductive group. We are interested in spherical homogeneous spaces $\mc G/\mc H$ for which there exists a parabolic subgroup $\mc P^-$ of $\mc G$ such that $\mc H\subseteq\mc P^-$ and $\mc P^-/\mc H$ is the $G\times G$-space $G$, for a suitable reductive group $G$. More precisely, there exists an epimorphism $\pi:\mc P^-\to G\times G$ such that $\mc H$ is the inverse image of the diagonal under $\pi$. In this case we can write $\mc G/\mc H=\mc G\times^{\mc P^-}G$: $\mc G/\mc H$ is induced from the homogeneous $G\times G$-space $G$. Recall that for a $\mc P^-$-variety $X$, $\mc G\times^{\mc P^-}G$ is the quotient of $\mc G\times X$ by the $\mc P^-$-action $h\cdot(g,x)=(gh^{-1},hx)$. The idea is that the well-known nice properties of (normal) $G\times G$-equivariant embeddings of $G$ should carry over to the $\mc G$-equivariant embeddings of $\mc G/\mc H$. The idea of our approach is based on \cite[Sect.~4]{DeC}, but our homomorphisms $\pi:\mc P^-\to G\times G$ are more general those in \cite[Sect.~4]{DeC}, see the end of Section~\ref{s.induction}. We also allow simple factors of the Levi of $\mc P^-$ to be in the kernel of $\pi$.

For results on $B\times B$-orbit closures in $G\times G$-equivariant embeddings of $G$ we refer to \cite{HeTh1} which uses the notions of strong and global $F$-regularity.

The main results of this paper are contained in Section~\ref{s.Ginduction}. In Proposition~\ref{prop.valuationcone} we give the divisor of a $\mc B$-semi-invariant function on $\mc G/\mc H$. As pointed out in Remark~\ref{rems.divisor}.1 one can easily extend this formula to one for the divisor on any $\mc G/\mc H$-embedding. In Theorem~\ref{thm.canonicaldivisor} we give the canonical divisor of any $\mc G/\mc H$-embedding. Proposition~\ref{prop.wfulpicard} gives the basic properties of the wonderful compactification of a suitable quotient of $\mc G/\mc H$. Our main result is Theorem~\ref{thm.Frobsplratres} which states that any $\mc G/\mc H$-embedding has an equivariant rational resolution by a toroidal embedding and that it is $\mc B$-canonically Frobenius split, compatible with the $\mc G$-orbit closures, the irreducible components of the complement of the open $\mc B$-orbit and those of the complement of the open $\mc B^-$-orbit. The results in Section~\ref{s.Ginduction} are extensions of known results for $G$-embeddings, see \cite{Rit2} and \cite[Ch.~6]{BriKu}. An important point in our general case is that the closed orbit in the wonderful compactification is not isomorphic to the full flag variety $\mc G/\mc B^-$. So we have to consider embeddings of an auxiliary homogeneous space $\mc G'/\mc H'$ and the wonderful compactification of a suitable quotient of this space. The smooth complete toroidal embeddings of these spaces are canonically split by a $(p-1)$-th power. These splittings are needed to prove rationality of certain resolutions using the Mehta-Van der Kallen Theorem \cite{MvdK}.

We now describe the contents of the paper in some more detail. In Section~\ref{s.prelim} we introduce the notation and we gather some general technical results, most of which occur, in some form, in the literature. In Section~\ref{s.normalorbitclosure} we show that all $G$-orbit closures in a spherical variety which is canonically Frobenius split are normal.
In Section~\ref{s.induction} we give some general results on induction from homogeneous spaces $G/H$ satisfying certain conditions. Some of these results occur in some form in the literature, but in this case they had to reformulated for our purposes. In Proposition~\ref{prop.induction} we show that the valuation cone of $\mc G/\mc H$ is the same as that of $G/H$. The main result in this section is Proposition~\ref{prop.toroidal} which shows that all toroidal embeddings of $\mc G/\mc H$ are induced from a toroidal $G/H$-embedding. In fact this gives a one-one correspondence.
Section~\ref{s.Ginduction} contains the main results of the paper that we described above. We also show that all $G\times G$-orbits in a normal reductive monoid with unit group $G$ occur in our class of homogeneous spaces.
In Section~\ref{s.examples} we describe our construction in the case of (circular) complexes (only of length $2$) and a reductive monoid $M$. We also determine the class group and the canonical divisor in the case of circular complexes and the monoid $M$. So in these cases we can determine when the variety is Gorenstein.

%

\section{Preliminaries}\label{s.prelim}
Throughout this paper $k$ is an algebraically closed field of prime characteristic $p$ and $G$ denotes a connected algebraic group over $k$, which we will assume to be reductive after the next lemma. The results in this paper that do not involve Frobenius splitting are also valid in characteristic $0$. Note that in this case normality of orbit closures and existence of rational resolutions holds for all spherical varieties. See \cite[Thm.~10]{Pop} and \cite[Prop.~3.5]{BriP}.

A {\it generalised valuation} of an irreducible variety $X$ is a pair $(Y,\nu)$, where $Y$ is an irreducible closed subvariety of $X$ and $\nu$ is a valuation (see \cite{ZS}) of the field $\mc O_{X,Y}/{\mf m_{X,Y}}$ with $\nu|_k$ trivial. Here $\mc O_{X,Y}$ and ${\mf m_{X,Y}}$ consist of the rational functions on $X$ that are defined at some point of  $Y$, respectively vanish on $Y$. Note that $\mc O_{X,Y}/{\mf m_{X,Y}}=k(Y)$, the field of rational functions on $Y$. We usually consider $\nu$ as a map on $\mc O_{X,Y}$ rather than $\mc O_{X,Y}/{\mf m_{X,Y}}$. Then $\mf m_{X,Y}$ is the inverse image of $\infty$. If $f\in\mc O_{X,Y}$, then we say that $\nu(f)$ is defined. Note that an ordinary valuation on $X$ is a generalised valuation $(Y,\nu)$ with $Y=X$. If $X$ is an irreducible $G$-variety and $f\in k(X)$, then we define $f^g$ by $f^g(x)=f(g\cdot x)$, whenever this is defined. The lemma below is included to make the proof of \cite[Prop.~4.1]{Bri2} more transparent, see Proposition~\ref{prop.limit} below. The problem with the proof in \cite[Prop.~4.1]{Bri2} is that the expression $f(g\lambda(t)x)$ is not defined for all $g\in G$. What is really needed is a good definition of $\nu_\lambda$.

\begin{lem}[cf.~{\cite[Lem.~10 and 11]{Sum1},\cite[Lem.~3.2]{LV},\cite[Lem.~1.4]{Knop}}]\label{lem.invval}
Let $X$ be an irreducible $G$-variety and let $(Y,\nu)$ be a generalised valuation of $X$. Then there exists a unique $G$-invariant generalised valuation $(\tilde Y,\tilde\nu)$ of $X$ such that
\begin{enumerate}[{\rm(1)}]
\item $\tilde Y=\overline{\bigcup_{g\in G}gY}$.
\item For every $f\in k(X)$ there exists a nonempty open subset $U_f$ of $G$ such that for all $g\in U_f$ we have $f^g\in\mc O_{X,Y}$ and $\tilde\nu(f)=\nu(f^g)$.
\end{enumerate}
\end{lem}

\begin{proof}
Let $\sigma:G\times X\to X$ be the action morphism and let $f\in\mc O_{X,\tilde Y}$. It suffices to show that for some nonempty open subset $U_f$ of $G$, we have $f^g\in\mc O_{X,Y}$ for all $g\in U_f$ and $\nu(f^g)$ is constant on $U_f$, since then we can define $\tilde\nu(f)$ to be this constant value. We have $f\circ\sigma\in\mc O_{G\times X,G\times Y}$. If $O$ is an affine open subset of $X$ intersecting $Y$, then $G\times O$ is affine and $\mc O_{G\times X,G\times Y}=\mc O_{G\times O,G\times Y\cap O}$. So we can write $f\circ\sigma=F_1/F_2$ with $F_1,F_2\in k[G]\ot\mc O_{X,Y}$ and $F_2|_{G\times Y}\ne0$. Put $F_{i,g}(x)=F_i(g,x)$. There exists $(g_0,y)\in G\times Y$ which is in the domains of $F_1$ and $F_2$ and with $F_2(g_0,y)\ne0$. It follows that for $g$ in some nonempty open set $U$ of $G$, $F_{i,g}\in\mc O_{X,y}$ and $F_{2,g}\notin\mf m_{X,y}$. 
We can write $F_i=\sum_jh_{ij}\ot f_{ij}$ with $h_{ij}\in k[G]$ and $f_{ij}\in\mc O_{X,Y}$. Let $V_i\subseteq \mc O_{X,Y}$ be the span of the $F_{i,g}$, $g\in U$. It is finite dimensional, since it is contained in the span of the $f_{ij}$. So $\nu$ must take a minimum value $a_i$ on $V_i$ and $a_2\ne\infty$. Then $\{f'\in V_i\,|\,\nu(f')>a_i\}$ is a proper subspace of $V_i$, or possibly empty when $i=1$. Now $g\mapsto F_{i,g}:U\to V_i$ is a morphism, so $U_i=\{g\in U\,|\,\nu(F_{i,g})=a_i\}$ is nonempty and open. We have $f^g=F_{1,g}/F_{2,g}$, so the set $U_f=U_1\cap U_2$ will do the job.
\end{proof}
From now on we will always assume that valuations are $\mb Q$-valued and discrete. So their value group is isomorphic to $\mb Z$ or trivial.
Let $\nu_t$ be the valuation of the field $k((t))$ of Laurent series with valuation ring $k[[t]]$ and $\nu_t(t)=1$. Let $X$ be an irreducible $G$-variety. For any $x(t)\in X(k((t)))$ there is a corresponding generalised valuation $\eta_{x(t)}=(Y,\eta_{x(t)})$ of $X$, where $Y$ is the irreducible closed subvariety of $X$ whose generic point is the image point of the morphism of schemes $\Spec(k((t)))\to X$ corresponding to $x(t)$, and $\eta_{x(t)}(f):=\nu_t(f(x(t)))$. We denote by $\nu_{x(t)}$ the $G$-invariant generalised valuation associated to $\eta_{x(t)}$ by Lemma~\ref{lem.invval}. Note that any morphism of varieties $k^\times\to X$ determines a point in $X(k((t)))$. In this case $Y$ is the closure of the image of this morphism.

In the remainder of this paper $G$ will be connected reductive and $H$ will be a closed subgroup scheme of $G$. We remind the reader that, as a topological space, $G/H$ (see \cite[III \S3]{DG} or \cite[I.5.6]{Jan}) can be identified with $G/H(k)$. More precisely, the canonical morphism $G/H(k)\to G/H$ is a homeomorphism. We denote the coset $H(k)\in G/H$ by $x$. The set of cocharacters of $G$ is denoted by $Y(G)$. If $\lambda\in Y(G)$, then we denote the generalised valuation of $G/H$ determined by $t\mapsto\lambda(t)\cdot x:k^\times\to G/H$ by $\eta_\lambda$. The corresponding $G$-invariant generalised valuation is an ordinary valuation and is denoted by $\nu_\lambda$.

\begin{rem}
Our definition of $\nu_\lambda$ is a simplified version of that in \cite{LV} and equivalent to the definition in \cite[Sect.~4.1]{Bri2}. The advantage of our definition is that one can compute $\nu_\lambda$ on a big class of rational functions. The valuations $\nu_\lambda$ in \cite{Rit1} (in the special case of $G$-embeddings) are not $G\times G$-invariant. Also not when $\lambda$ is dominant or anti-dominant, as one can see by checking the case $\GL_2$.
\end{rem}

By a {\it $G/H$-embedding} we will always mean a normal irreducible $G$-variety in which $G/H$ embeds $G$-equivariantly as a $G$-stable open subset. The irreducible components of the complement of $G/H$ in a $G/H$-embedding $X$ are called the {\it boundary divisors} or {\it $G$-stable prime divisors} of $X$. A generalised valuation of $G/H$ is, of course, also a generalised valuation of any $G/H$-embedding. If $X$ is a $G/H$-embedding, $\lambda\in Y(G)$ and $\lim_{t\to 0}\lambda(t)\cdot x$ exists and equals $y\in X$, then $\eta_\lambda(f)$ is defined and $\ge0$ for all $f\in\mc O_{X,y}$ and it is $>0$ for all $f\in\mf m_{X,y}$.

We now fix some standard notation which we will use throughout the paper. The group $B$ is always a Borel subgroup of $G$ and $T$ is a maximal torus of $B$. The unipotent radical of $B$ is denoted by $U$. We denote the character and cocharacter group of $T$ by $X(T)$ and $Y(T)$. For induced modules and line bundles we will follow the notation in \cite{Jan}, except that we denote the $B$ and $B^+$ from \cite{Jan} by $B^-$ and $B$ and similar for their unipotent radicals. We will use similar notation and conventions for the reductive group $\mc G$ in Section~\ref{s.induction}. For the general theory of reductive groups and their representations we refer to \cite{Bo} and \cite{Jan}. The sheaf $U\mapsto\{f\in k(X)^\times\,|\,((f)+D)|_U\ge0\}\cup\{0\}$ corresponding to a divisor $D$ on a normal irreducible variety $X$ is denoted by $\mc O_X(D)$ and, in case $X$ is smooth, the corresponding line bundle is denoted by $\mc L_X(D)$. If $X$ is smooth, then the canonical bundle and its sheaf of sections are both denoted by $\omega_X$. The canonical sheaf $\omega_X$ of a normal irreducible variety $X$ is defined as the push-forward of the canonical sheaf of the smooth locus of $X$, see \cite[Rem.~1.3.12]{BriKu}. It is of the form $\mc O_X(K_X)$ for a Weil divisor, called the canonical divisor, $K_X$ on $X$ which is determined up to linear equivalence.


In the remainder of this paper we assume that $G/H$ is spherical with $B\cdot x$ open in $G/H$. For the general theory of spherical varieties we refer to \cite{Bri2}, \cite{Knop} and \cite{BriI}. A $G/H$-embedding is called {\it simple} if has a unique closed $G$-orbit. Simple embeddings are described by a single coloured cone. A $G/H$-embedding is called {\it toroidal} if its coloured fan has no colours, that is, if no $B$-stable prime divisor which intersects $G/H$ (i.e. which is not $G$-stable) contains a $G$-orbit. An {\it elementary embedding} is a simple toroidal embeddings whose cone is a half ray in the valuation cone of $G/H$. For a toroidal $G/H$-embedding $X$ we denote by $X_0$ the complement in $X$ of the union of the $B$-stable prime divisors that intersect $G/H$. Note that every $G$-orbit of $X$ intersects $X_0$. If $X$ is simple, then $X_0$ is affine by \cite[Thm.~2.1]{Knop}. We denote the closure of $T\cdot x$ in $X_0$ by $\ov{T\cdot x}^0$ (so it has to be clear what $X$ is). We will say that the {\it local structure theorem} holds for $X$ (relative to $x$, $T$ and $B$) if there exists a parabolic subgroup $Q$ of $G$ containing $B$ such that, with $M$ the Levi subgroup of $Q$ containing $T$, we have

\smallskip
The variety $\ov{T\cdot x}^0$ is $M$-stable, the derived group $DM$ acts trivially on it, and the action of $G$ induces an isomorphism $R_u(Q)\times \ov{T\cdot x}^0\stackrel{\sim}{\to}X_0$.
\smallskip

Note that the $T$-orbit of $x$ is the same as the orbit of $x$ under the action of the connected centre of $M$.
Note, furthermore, that the local structure theorem for a toroidal embedding of $G/H$ implies the local structure theorem for $G/H$ itself. This determines $Q$ uniquely, because $M$ is the biggest Levi subgroup in $G$ with $T\subseteq M\subseteq H$. Actually $Q$ can be characterised as the stabiliser of the open $B$-orbit $B\cdot x$ or as the stabiliser of $BH(k)$ (for the action by left multiplication), but we will not make use of this. 

\begin{prop}[{\cite[Prop.~4.1]{Bri2}}]\label{prop.limit}
Let $X$ be a $G/H$-embedding, let $Y$ be a boundary divisor and let $\nu$ be the corresponding valuation of $k(G/H)$.
\begin{enumerate}[{\rm(i)}]
\item If $\lambda\in Y(T)$ is such that $y=\lim_{t\to 0}\lambda(t)\cdot x$ exists and lies in the open $B$-orbit of $Y$, then $m\nu=\nu_\lambda$ for some $m\ge1$ and the image of $m\nu$ in the valuation cone is the image of $-\lambda$.
\item If the local structure theorem holds for the elementary $G/H$-embedding associated to $\nu$, then there exists a $\lambda$ as in (i).
\end{enumerate}
\end{prop}

\begin{proof}
(i).\ We have that $\lim_{t\to 0}g\lambda(t)\cdot x$ exists and equals $g\cdot y\in Y$ for all $g\in G$. So if $f\in\mc O_{X,Y}$, then $\eta_\lambda(f^g)\ge0$ for all $g$ in the nonempty open subset $U_f=\{g\in G\,|\,g\cdot y\in {\rm Dom}(f)\}$ of $G$. So $\nu_\lambda\ge0$ on $\mc O_{X,Y}$, by Lemma~\ref{lem.invval}. Similarly, $\nu_\lambda>0$ on $\mf m_{X,Y}$. As in \cite{Bri2} we conclude that $\nu_\lambda=m\nu$ for some $m\ge1$. Let $\chi\in X(T)$ be the weight of a nonzero $B$-semi-invariant $f\in k(G/H)$. Then $f(\lambda(t)\cdot x)=t^{-\la\chi,\lambda\ra}f(x)$. Now assume $\nu(f)\ge0$. Then $f\in\mc O_\nu=\mc O_{X,Y}$. Since $f$ is a $B$-semi-invariant we get that $f$ is defined at $y$. So we must have $\la\chi,\lambda\ra\le 0$. As in \cite{Bri2} we conclude that the image of $m\nu$ in the valuation cone is the image of $-\lambda$.\\
(ii).\ Define $X_{Y,B}\subseteq X$ as in \cite[Sect.~2.2]{Bri2} 
or \cite[Sect.~2]{Knop} (there denoted $X_0$). Then the union of the $G$-conjugates of $X_{Y,B}$ is an open subembedding of $X$ which is the elementary embedding corresponding to $\nu$. So the assertion follows from the theory of toric varieties.
\end{proof}

Note that it follows from Proposition~\ref{prop.limit} that, if the local structure theorem holds for the elementary $G/H$-embeddings, all $G$-invariant valuations are of the form $\nu_\lambda$.

\begin{rem}\label{rem.danger}
The orbit $T\cdot x$ has a unique structure of a torus such that the $T$-orbit map is a morphism of tori. The $T$-weight of a character $\chi\in X(T\cdot x)\subseteq k(G/H)^U$ is the negative of its image in $X(T)$ when composing with the orbit map. We will identify ${\mb Q}\ot X(T\cdot x)$ with the space of weights by identifying $\chi\in X(T\cdot x)$ with the $T$-weight of $\chi$ (or of its $U$-invariant extension to $G/H$). Of course this also yields an identification of $\mb Q\ot Y(T\cdot x)$ with the dual of the space of weights. The image of $\lambda\in Y(T)$ in the dual of the space of weights is then the negative of the composite of $\lambda$ with the orbit map.
\end{rem}

In the case of the $G\times G$-space $G$, $G\times G$ acting by $(g_1,g_2)\cdot h=g_1hg_2^{-1}$, we take the unit of $G$ as the base point $x$, $B^-\times B$ as Borel subgroup and $T\times T$ as maximal torus. Then $(T\times T)\cdot x=T$. Furthermore, by \cite[Prop.~6.2.4]{BriKu}, the local structure theorem holds for all toroidal $G$-embeddings and the parabolic $Q$ is $B^-\times B$. Furthermore the weight of $\chi\in X(T)\subseteq k(G)^{U^-\times U}$ is $(-\chi,\chi)$. 
For $\lambda\in Y(G)$, we define $\nu_\lambda=\nu_{(\lambda,0)}$, a $G\times G$-invariant valuation on $G$. Note that it follows  immediately from the definition of $\nu_\lambda$ that $\nu_\lambda=\nu_{w(\lambda)}$ for $\lambda\in Y(T)$ and any $w$ in the Weyl group. From this, \cite[Cor.~5.3]{Knop} and Proposition~\ref{prop.limit} the reader can easily deduce the following corollary. Alternatively, one can deduce it as in \cite[Ex.~4.1]{Bri2}. 
\begin{cornn}[{cf. \cite[Prop.~2]{V}, \cite[Ex.~4.1]{Bri2}, \cite[Prop.~9]{Rit1}}]\
\begin{enumerate}[{\rm(i)}]
\item Let $\lambda\in Y(T)$, let $X$ be the elementary $G$-embedding corresponding to $\nu_\lambda$ and let $Y$ be the closed $G$-orbit in $X$. Then $\lim_{t\to 0}\lambda(t)$ exists and lies in the open $B^-\times B$-orbit of $Y$ if and only if $\lambda$ is anti-dominant. Furthermore, the image of $\nu_\lambda$ in $\mb Q\ot Y(T)$ is the anti-dominant Weyl group conjugate of $\lambda$.
\item The valuation cone of $G$ is the anti-dominant Weyl chamber in $\mb Q\ot Y(T)$.
\end{enumerate}
\end{cornn}
\begin{rem}\label{rem.wful}
In \cite{BriKu}, \cite{Rit2} and \cite{Str} there is a sign mistake in the statement of the local structure theorem for the wonderful compactification of the adjoint group $G_{\rm ad}$. Let $\alpha_1,\ldots,\alpha_r$ be the simple roots of $G$. If we form ${\bf X}_0$ according to $B^-\times B$ (not $B\times B^-$ as claimed in these sources), then $\ov{T\cdot x}^0$ is isomorphic to $\mb A^r$ via $t\mapsto(\alpha_1^{-1}(t),\ldots,\alpha_r^{-1}(t))$. If we choose $B\times B^-$ as a Borel subgroup, then the isomorphism $\ov{T\cdot x}^0\cong\mb A^r$ is given by $t\mapsto(\alpha_1(t),\ldots,\alpha_r(t))$. In \cite[Ch.~6]{BriKu} the cause is that the action of $G\times G$ on $\End(M)\cong M^*\ot M$ is given by letting the left copy of $G$ act on $M^*$ and the right copy on $M$. This amounts to the action $(g_1,g_2)\cdot \varphi=g_2\varphi g_1^{-1}$ (composition of endomorphisms of $M$). So the embedding $G_{\rm ad}\hookrightarrow\bf X$ would only be equivariant if we would define the action of $G\times G$ on $G_{\rm ad}$ by the analogous formula. This doesn't really affect the other results in \cite[Ch.~6]{BriKu}; things can be corrected by swapping $B$ and $B^-$ at the appropriate places.
This may be related to the claim in \cite{Rit1} and \cite{Bri2} that the valuation cone of $G$ is the anti-dominant cone when we use $B\times B^-$ as Borel subgroup. By the above corollary, this is the case when we use $B^-\times B$ as Borel subgroup. The latter difference could also have been caused by another convention to identify $\mb Q\ot X(T)$ with the space of weights, see Remark~\ref{rem.danger}.
\end{rem}

Recall that a {\it resolution (of singularities)} of an irreducible variety $X$ is a smooth irreducible variety $\tilde X$ together with a proper birational morphism $\varphi:\tilde X\to X$. Note that if $X$ is normal we have $\varphi_*(\mc O_{\tilde X})=\mc O_X$, see e.g. \cite[II.14.5]{Jan} (the projectivity condition can be weakened to properness, see \cite[III.3.2.1]{Gro2}). 
As in \cite{BriKu}, we call a line bundle {\it semi-ample} if some positive power of it is generated by its global sections. As in \cite[Sect.~3.3]{BriKu} (following Kempf) we define a morphism $f:X\to Y$ of varieties to be {\it rational} if the direct image under $f$ of the structure sheaf of $X$ is that of $Y$ and if the higher direct images are zero, that is, if $f_*(\mc O_X)=\mc O_Y$ and $R^if_*(\mc O_X)=0$ for $i>0$. The lemma below can be proved with the arguments from Lem.~6.2.8 and Cor.~6.2.8 in \cite{BriKu}.

\vfill\eject

\begin{lem}[{\cite[Lem.~6.2.8]{BriKu}}]\label{lem.vanishing}\
\begin{enumerate}[{\rm(i)}]
\item If $X$ is a smooth projective $G/H$-embedding which is Frobenius split, compatible with the complement of the open $B$-orbit, then $H^i(X,\mc L)=0$ for any $i>0$ and any semi-ample line bundle $\mc L$ on $X$.
\item If $X$ is as in (i) and $f:X\to Y$ is a morphism of projective varieties, then $R^if_*(\mc O_X)=0$ for all $i>0$.
\item If all toroidal $G/H$-embeddings are Frobenius split, compatible with the complement of the open $B$-orbit, then any morphism $\varphi:\tilde X\to X$ of $G/H$-embeddings which is a resolution with $\tilde X$ toroidal and quasi-projective, is a rational morphism.
\end{enumerate}
\end{lem}
\begin{rem}
If $\varphi:\tilde X\to X$ is a resolution as in (iii) above, then $\varphi$ is a proper and quasi-projective morphism ($\tilde X$ is quasi-projective) and therefore, by \cite[Thm.~5.5.3]{Gro}, it is a projective morphism in the sense of \cite{Gro}. A projective morphism in the sense of \cite{Gro} to a quasi-projective variety is also projective in the sense of \cite{H2}. See \cite[p.~103 and Prop.~II.7.10]{H2}.
\end{rem}

The next lemma will be needed in the Section~\ref{s.induction}. It can be proved using \cite[Thm~3.3.4]{BriKu}, the Grothendieck spectral sequence for the direct image of the composite $G/B\to G/P\to G/Q$ and $\mc O_{G/B}$, and \cite[Prop.~III.9.3]{H2}.

\begin{lem}\label{lem.flagbundle}
Let $P$ and $Q$ be parabolic subgroups of $G$ with $P\subseteq Q$ and let $X$ be a $Q$-variety. Then the canonical morphism $G\times^PX\to G\times^QX$ is projective and rational.
\end{lem}
The following lemma is well-known. It was used, for example, in \cite{MT1} and \cite{MT2}. We will need it in Section~\ref{s.Ginduction}. For simplicity we added the assumption that $f$ is smooth.

\begin{lem}[{\cite[Cor.~VII.3.4]{H1}, \cite[p.~49]{KKMS}}]\label{lem.duality}
Let $X$ and $Y$ be smooth varieties and let $f:X\to Y$ be a proper smooth rational morphism of relative dimension $d$. Then
$$R^if_*\omega_X=
\begin{cases}
\omega_Y&\text{if\ }i=d,\\
0&\text{otherwise.}
\end{cases}
$$
\end{lem}

\section{Normality of orbit closures}\label{s.normalorbitclosure}
Recall from \cite[Sect.~1.3]{Do} that a {\it good pair of varieties} is a pair $(X,Y)$ of affine varieties with $Y$ a closed subvariety of $X$ such that $k[X]$ and the vanishing ideal of $Y$ in $k[X]$ have a good filtration. In this case the algebra $k[Y]$ also has a good filtration.
\begin{prop}\label{prop.normalorbitclosure}
Let $X$ be a (normal) affine spherical $G$-variety. Let $Y$ be a $G$-orbit closure in $X$. If $(X,Y)$ is a good pair of varieties, then $Y$ is normal.
\end{prop}

\begin{proof}
We follow the arguments in \cite[Prop.~3.5]{BriP} combined with properties of good filtrations. Let $T$ and $U$ be a maximal torus and the unipotent radical of a Borel subgroup of $G$. Since $k[Y]$ has a good filtration, it suffices, by \cite[Thm.~17]{Gr1}, to show that $k[Y]^U$ is normal. Note that it is finitely generated by \cite[Thm.~16.4]{Gr2}. Since $X$ is normal, $k[X]^U$ is normal. Furthermore, since $X$ is spherical, the variety $\Maxspec(k[X]^U)$ is a normal $T$-variety with a dense $T$-orbit. If $I$ is the ideal of functions in $k[X]$ that vanish on $Y$, then $k[Y]^U\cong k[X]^U/I^U$, since $(X,Y)$ is a good pair of varieties. Here we used that the $U$-fixed point functor is exact on short exact sequences of modules with a good filtration (\cite[II.2.13, 4.13]{Jan}). Thus $\Maxspec(k[Y]^U)$ identifies with a $T$-orbit closure in $\Maxspec(k[X]^U)$. So it is normal, by the general theory of toric varieties.
\end{proof}

\begin{cornn}
Let $X$ be a spherical $G$-variety and let $Y$ be a $G$-orbit closure in $X$. If $X$ is $B$-canonically Frobenius split compatible with $Y$, then $Y$ is normal.
\end{cornn}

\begin{proof}
We will reduce to the affine case with similar arguments as in 6.2.8 and 6.2.14 in \cite{BriKu}. By a theorem of Sumihiro (\cite{Sum1},\cite{Sum2},\cite{KKLV}) $X$ is covered by $G$-stable open subsets which are $G$-equivariantly isomorphic to locally closed subsets of projective $G$-varieties. So we may assume that $X$ is quasi-projective. Then we may replace $X$ by the normalisation of its closure in a projective $G$-variety and assume that it is projective. Note that the original $X$ will embed as an open subset in this normalization, since it is normal. By another theorem of Sumihiro ({\it loc. cit.}) there exists a very ample $G$-linearised invertible sheaf $\mc L$ on $X$. After replacing $\mc L$ by a suitable power we may assume that $X$ is projectively normal for the corresponding projective embedding (see \cite[Ex.~II.5.13, 14]{H2}). This means that the corresponding affine cone $\hat X:=\Maxspec(R(X,\mc L))$, where $R(X,\mc L)=\bigoplus_{m\ge0}H^0(X,\mc L^m)$, is normal. This is a spherical $\hat G:=G\times\mb G_m$-variety. Furthermore, the restrictions $H^0(X,\mc L^m)\to H^0(Y,\mc L^m)$ are surjective for all $m\ge0$, by \cite[Thm.~1.2.8]{BriKu}. So $\hat Y:=\Maxspec(R(Y,\mc L))$ identifies with an irreducible closed $\hat G$-stable subvariety of $\hat X$, i.e. with a $\hat G$-orbit closure in $\hat X$. The cone $\hat X$ is $\hat B$-canonically Frobenius split by \cite[Lem.~4.1.13]{BriKu} and from the definition of this splitting it follows easily that $\hat Y$ is compatibly split. Now $(\hat X,\hat Y)$ is a good pair of varieties by \cite[Ex.~4.2.E.2]{BriKu} applied to the trivial bundle, so $\hat Y$ is normal by Proposition~\ref{prop.normalorbitclosure}. Thus $Y$ is projectively normal for the projective embedding defined by $\mc L$ and therefore also normal.
\end{proof}

\begin{rem}
The above result applies for example to all (normal) equivariant embeddings of the homogeneous spaces studied in Section~\ref{s.Ginduction}. See Theorem~\ref{thm.Frobsplratres}. In the case of $G\times G$-equivariant embeddings of $G$ much stronger results were obtained by He and Thomsen valid for all $B\times B$-orbit closures, see \cite[Cor.~8.4]{HeTh1}.
\end{rem}

\section{Induction}\label{s.induction}

Throughout this section $\mc G$ is a reductive algebraic group, $\mc P^-$ is a parabolic subgroup of $\mc G$, $\mc H$ is a closed subgroup scheme of $\mc P^-$, $G$ is a reductive group, $H$ is a closed subgroup scheme of $G$ and $\pi:\mc P^-\to G$ is an epimorphism such that
\begin{enumerate}[(\rm{A}1)]
\item $G/H$ is spherical.
\item The isogeny $\mc P^-/\Ker(\pi)\to G$ induced by $\pi$ is central (cf. \cite[V.22]{Bo}).
\end{enumerate}

We will introduce several subgroups of $G$ and $\mc G$ and we encourage the reader to look at Section~\ref{ss.detcirc} to see what they are in some explicit examples. Let $B$ be a Borel subgroup of $G$ such that $BH(k)$ is open in $G$, let $T$ be a maximal torus of $B$ and let $B^-$ be the Borel subgroup of $G$ which is opposite to $B$ relative to $T$. Furthermore, let $\mc T$ be a maximal torus of $\mc P^-$ with $\pi(\mc T)=T$, let $\mc B^-$ be a Borel subgroup of $\mc P^-$ containing $\mc T$ with $\pi(\mc B^-)=B^-$ and let $\mc P$ and $\mc B$ be the parabolic and Borel subgroups of $\mc G$ that are opposite to $\mc P^-$ and $\mc B^-$. Now let $\mc K$ be the Levi subgroup of $\mc P$ containing $\mc T$ and let $\mc L$ be the subgroup of $\mc K$ generated by $\mc T$ and the simple factors of $\mc K$ that do not lie in the kernel of $\pi$. Put $\mc B_{\mc L}=\mc B\cap\mc L$.

\smallskip
\begin{enumerate}[(\rm{A}3)]
\item The scheme theoretic inverse image $\ov{\mc H}$ of $H$ under $\pi$ is generated by $\mc H$ and a closed subgroup scheme of $\mc T$ that normalises $\mc H$.
\end{enumerate}
\smallskip

By (A3) we have a morphism $\mc P^-/\mc H\to G/H$ intertwining the actions of $\mc P^-$ and $G$, and $R_u(\mc P^-)\subseteq\mc H(k)$. Note also that $\mc H$ contains the simple factors of $\mc K$ that lie in the kernel of $\pi$. 
As we will see in the next lemma, $\mc G/\mc H$ is a spherical homogeneous space. We will study its equivariant embeddings. Denote the coset $H(k)\in G/H$ by $y$, denote the coset $\mc H(k)\in \mc P^-/\mc H\subseteq\mc G/\mc H$ by $x$ and denote the coset $\ov{\mc H}(k)\in\mc G/\ov{\mc H}$ by $\ov x$. In this section we will usually identify $\ov x$ with $y$.

\begin{lem}\label{lem.inducedBstabledivisors}
Let $Y$ be a $\mc P^-$-variety with an open $\mc B\cap\mc K$-orbit. Then $\mc  G\times^{\mc P^-}Y$ has an open $\mc B$-orbit. Furthermore, the $\mc B$-stable prime divisors of $\mc  G\times^{\mc P^-}Y$ are the inverse images (or pull-backs) of the $\mc B$-stable prime divisors of $\mc G/\mc P^-$ together with the $\mc B$-stable prime divisors that intersect $Y$. These intersections are the $\mc B\cap\mc K$-stable prime divisors of $Y$.
\end{lem}

\begin{proof}
The variety $\mc  G\times^{\mc P^-}Y$ is a locally trivial fibration over $\mc G/\mc P^-$ and $Y$ is the inverse image of the coset $\mc P^-$. In fact the action of $\mc G$ induces an isomorphism $R_u\mc P\times Y\cong\mc PY$. Since $\mc B=R_u\mc P(\mc B\cap\mc K)$, this shows that $\mc  G\times^{\mc P^-}Y$ has an open $\mc B$-orbit. Let $D$ be a $\mc B$-stable prime divisor of $\mc  G\times^{\mc P^-}Y$. Then $D$ intersects $Y$ if and only if it intersects $\mc BY=\mc PY$ if and only if its open $\mc B$-orbit intersects $Y$ if and only if its image in $\mc G/\mc P^-$ is dense. Assume this is the case. Then $D\cap\mc PY\cong R_u(P)\times(D\cap Y)$. So $D\cap Y$ is a $\mc B\cap\mc K$-stable prime divisor of $Y$. Now assume that the image of $D$ in $\mc G/\mc P^-$ is not dense. Then the closure of this image must be a $\mc B$-stable prime divisor of $\mc G/\mc P^-$. The inverse image in $\mc  G\times^{\mc P^-}Y$ of this closure is closed, irreducible 
and of pure codimension $1$, so it must be equal to $D$.
\end{proof}
For more details on the $B$-orbits after parabolic induction from a Levi subgroup, see \cite[Sect.~1.2]{Bri3}, \cite[Lemma~6]{Bri4}.

Finally, we assume that the local structure theorem holds for $G/H$, i.e. there exists a parabolic subgroup $Q$ of $G$ containing $B$ such that, with $M$ the Levi subgroup of $Q$ containing $T$, we have

\smallskip
\begin{enumerate}[(\rm{A}4)]
\item The derived group $DM$ is contained in $H(k)$ (i.e. it fixes $y$) and the action of $G$ induces an isomorphism $R_u(Q)\times T\cdot y\stackrel{\sim}{\to}Q\cdot y$.
\end{enumerate}
\smallskip

Let $\mc Q$ be the inverse image of $Q$ under the homomorphism $\mc P\to\mc K\stackrel{\pi}{\to}G$ and let $\mc M$ be the Levi subgroup of $\mc Q$ containing $\mc T$. Note that $\mc Q=(\mc Q\cap\mc K)R_u(\mc P)$ and that $\mc Q\cap\mc K$ is a parabolic subgroup of $\mc K$ with $\mc M$ as a Levi subgroup. Note furthermore that $\mc M$ contains the simple factors of $\mc K$ that lie in the kernel of $\pi$. 
Denote the scheme theoretic intersection of $T$ and $H$ by $H_T$ and define $\mc H_{\mc T}$ analogously.

\begin{prop}\label{prop.induction}
Assume (A1)-(A4) and let $\mc Q$ and $\mc M$ be as defined above.
\begin{enumerate}[{\rm(i)}]
\item We have $D\mc M\subseteq\mc H(k)$ and the action of $\mc G$ induces an isomorphism\\ $R_u(\mc Q)\times\mc T\cdot x\stackrel{\sim}{\to}\mc Q\cdot x$.
\item The weight lattice of $\mc G/\mc H$ is isomorphic to $X(\mc T/\mc H_{\mc T})\subseteq X(\mc T)$.
\item The valuation cone of $\mc G/\mc H$ is the inverse image of the valuation cone of $G/H$ under the map $\mb Q\ot Y(\mc T/\mc H_{\mc T})\twoheadrightarrow\mb Q\ot Y(T/H_T)$.
\item The $\mc B$-stable prime divisors of $\mc G/\mc H$ are the inverse images of the $\mc B$-stable prime divisors of $\mc G/\mc P^-$ together with the $\mc B$-stable prime divisors that intersect $\mc P^-/\mc H$. These intersections are the $\mc B_{\mc L}$-stable prime divisors of $\mc P^-/\mc H$, they are the pull-backs of the $B$-stable prime divisors of $G/H$.
\end{enumerate}
\end{prop}

\begin{proof}
(i).\ We have $\mc G/\mc H\cong\mc G\times^{\mc P^-}\mc P^-/\mc H$ and the action of $\mc G$ induces an isomorphism $R_u\mc P\times\mc P^-/\mc H\cong\mc P\cdot x$. Furthermore, $D\mc M\subseteq\pi^{-1}(DM)\subseteq\mc H(k)\mc Z(k)$, where $\mc Z$ is given by (A3). So $D\mc M=DD\mc M\subseteq\mc H(k)$. So it suffices to show that the action of $\mc Q\cap\mc K$ induces an isomorphism $R_u(\mc Q\cap\mc K)\times\mc T\cdot x\stackrel{\sim}{\to}(\mc Q\cap\mc K)\cdot x$. Next, we may replace $\mc Q\cap\mc K$ by $\mc Q\cap\mc L$. Then, by (A2), the epimorphism $\pi:\mc L\to G$ is central (the epimorphism $\mc L\to \mc P^-/\Ker(\pi)$ is central), so it induces an isomorphism $R_u(\mc Q\cap\mc K)=R_u(\mc Q\cap\mc L)\stackrel{\sim}{\to}R_u(Q)$. Now the result follows from (A4), using the morphism $\mc P^-/\mc H\to G/H$.\\
(ii).\ This follows immediately from (i) and the fact that $\mc T\cdot x\cong \mc T/\mc H_{\mc T}$.\\
(iii).\ We have $k(\mc P^-/\mc H)\cong k(\mc G/\mc H)^{R_u(\mc P)}\subseteq k(\mc G/\mc H)$. Now $R_u(P^-)$ acts trivially on $k(\mc P^-/\mc H)$, since $R_u(\mc P^-)\subseteq\mc H$ and $R_u(\mc P^-)\trianglelefteq\mc P^-$. So if $f\in k(\mc G/\mc H)$ is $R_u(\mc P)$-fixed, then it is also $R_u(\mc P^-)$-fixed. Now let $\nu$ be a valuation of $k(\mc G/\mc H)$, let $f\in k(\mc G/\mc H)^{R_u(\mc P)}$ and assume that $\nu(f^g)$ is constant for $g$ in some nonempty open subset $U_f$ of $\mc G$. Then the same holds for $g$ in $R_u(\mc P)U_f'R_u(\mc P^-)$, where $U_f'$ is a nonempty open subset of $\mc L$. So if we extend an $\mc L$-invariant valuation of $k(\mc P^-/\mc H)$ to $k(\mc G/\mc H)$ and then make it invariant using Lemma~\ref{lem.invval}, then the resulting valuation will restrict to our original valuation. Thus the restriction to $k(\mc G/\mc H)^{R_u(\mc P)}$ induces a bijection between the $\mc G$-invariant valuations of $k(\mc G/\mc H)$ and the $\mc L$-invariant valuations of $k(\mc P^-/\mc H)$.

Let $\mc Z$ be a closed subgroup scheme of $\mc T$ as in (A3). Then $k(G/H)=k(\mc P^-/\mc H\mc Z)=k(\mc P^-/\mc H)^{\mc Z}$, where the action of $\mc Z$ comes from right multiplication. So we can finish by showing that a valuation of $k(\mc P^-/\mc H)$ is $\mc L$-invariant, if this holds for its restriction to $k(G/H)$. Let $\mc H_{\mc L}$ be the scheme theoretic intersection of $\mc H$ and $\mc L$. Then $\mc P^-/\mc H=\mc L/\mc H_{\mc L}$. We follow the argument of the proof of \cite[Satz~8.1.4]{Knop2}, but use the notation of \cite[Lemma~5.1]{Knop}. See also \cite{Bri2}, proof of Thm.~4.3, for similar arguments. Let $g_1,\ldots,g_s\in k(\mc L/\mc H_{\mc L})^{(\mc B_{\mc L})}\subseteq k(\mc T/\mc H_{\mc T})$ and $h\in k[\mc L]^{(\mc B_{\mc L}\times\mc H_{\mc L}(k))}$ with $f_i=hg_i\in k[\mc L]$. After replacing $h$ by a suitable power of it we may assume that $h\in k[\mc L]^{(\mc B_{\mc L}\times\mc H_{\mc L})}$. Let $M_i$ be the $\mc L$-module generated by $f_i$. Since $\mc B_{\mc L}\cdot x$ is open in $\mc L/\mc H_{\mc L}$ and $\mc Z\subseteq B_{\mc L}$ any $B_{\mc L}$-semi-invariant in $k(\mc L/\mc H_{\mc L})$ is also a $\mc Z$-semi-invariant. 
So the elements of $M_i$ are all $\mc Z$-semi-invariants with the same $\mc Z$-character as $f_i$. Furthermore, if $f'\in(M_1\cdots M_s)^{(\mc B_{\mc L})}$, then $f'/{h^s}$ has the same $\mc Z$-character as the product of the $g_i$ (Note that these remarks also apply with $\mc Z$ replaced by $\mc H_{\mc L}$). So $u=f'/(h^s\prod_ig_i)$ has trivial $\mc Z$-character and appears therefore in $k(G/H)^{(B)}\subseteq k(T/H_T)$ by (A4). Now let $\nu$ be a valuation of $k(\mc L/\mc H_{\mc L})$ whose restriction to $k(G/H)$ is $G$-invariant. By a standard argument using Lemma~\ref{lem.invval} (for ordinary valuations) this restriction has an $\mc L$-invariant extension $\ov\nu$ to $k(\mc L/\mc H_{\mc L})$. By \cite[Lemma~5.1]{Knop} $\ov\nu(u)\ge0$ for all possible choices of the $g_i$, $h$ and $f'$. But then the same must hold for $\nu$, since $u\in k(G/H)$ and $\nu$ and $\ov\nu$ agree on $k(G/H)$. Now $\nu$ is $\mc L$-invariant, again by \cite[Lemma~5.1]{Knop}.\\
(iv).\ Except for the final statement this follows from Lemma~\ref{lem.inducedBstabledivisors}. Similar arguments as above show that every $\mc B_{\mc L}\times\mc H_{\mc L}$-semi-invariant in $k(\mc L)$ is also a $\mc B_{\mc L}\times\ov{\mc H}_{\mc L}$-semi-invariant. This implies the final assertion.
\end{proof}

\begin{prop}\label{prop.toroidal}
Assume (A1)-(A4).
\begin{enumerate}[{\rm(i)}]
\item Let $Y$ and $Z$ be toroidal embeddings of $\mc L/\mc H_\mc L$ and $\mc L/\ov{\mc H}_\mc L=G/H$ respectively. Assume the canonical morphism $\mc L/\mc H_\mc L\to\mc L/\ov{\mc H}_\mc L$ extends to a morphism $Y\to Z$. If the local structure theorem holds for $Z$, then it holds for $Y$ and $\mc G\times^{\mc P^-}Y$. If $Y$ is complete or quasi-projective, then so is $\mc G\times^{\mc P^-}Y$.
\item Assume that $\ov{\mc H}=\mc H$. The toroidal embeddings of $\mc G/\mc H$ are all of the form $\mc G\times^{\mc P^-}Y$, where $Y$ is a toroidal embedding of $G/H=\mc P^-/\mc H$. This determines a $1$-$1$-correspondence. Furthermore, the local structure theorem holds for $Y$ if and only if it holds for $\mc G\times^{\mc P^-}Y$.
\end{enumerate}
\end{prop}

\begin{proof}
(i).\ The first assertion follows as in the proof of Proposition~\ref{prop.induction}(i). The assertion about completeness is standard. Now assume $Y$ is quasi-projective. Then it has an ample $\mc L$-linearised line bundle $L$. From the local triviality of $pr:\mc G\times^{\mc P^-}Y\to\mc G/\mc P^-$ we deduce that $\mc G\times^{\mc P^-}L$ is ample relative to $pr$. Since $\mc G/\mc P^-$ is projective, it follows from \cite[Prop.~4.6.13(ii)]{Gro} that $\mc G\times^{\mc P^-}Y$ is quasi-projective. See \cite[Prop.~6.2.3(iv)]{BriKu} for similar arguments.\\
(ii).\ By Proposition~\ref{prop.induction}(iii) the valuation cone of $\mc G/\mc H$ is the same as that of $G/H$. So the toroidal embeddings of these spaces are described by the same data. Furthermore it is clear that if $Y$ is the toroidal $G/H$-embedding corresponding to a fan without colours, then $\mc G\times^{\mc P^-}Y$ is the toroidal $\mc G/\mc H$-embedding corresponding to the same fan. The final assertion is clear.
\end{proof}

\begin{cornn}
Under the assumptions of Proposition~\ref{prop.toroidal}(ii), if $Y$ is $B$-canonically Frobenius split, then $\mc G\times^{\mc P^-}Y$ is $\mc B$-canonically Frobenius split.
\end{cornn}

\begin{proof}
If $Y$ is $B$-canonically Frobenius split, then it is also $B^-$-canonically Frobenius split, by \cite[Prop.~4.1.10]{BriKu}. Then it is clearly also $\mc B^-$-canonically Frobenius split (as a $\mc P^-$-variety), so, by a result of Mathieu (see \cite[Prop.~5.5]{Mat} or \cite[Thm.~4.1.17]{BriKu}), $\mc G\times^{\mc B^-}Y$ is $\mc B^-$-canonically Frobenius split. But then $\mc G\times^{\mc P^-}Y$ is $\mc B^-$-canonically (and therefore also $\mc B$-canonically) Frobenius split by Lemma~\ref{lem.flagbundle} and \cite[Ex.~4.1.E.3]{BriKu}.
\end{proof}

\begin{rem}
Mathieu's result also gives some compatible splittings.
It is not clear, however, how one can get a splitting such that the $\mc B$-stable prime divisors and the $\mc B^-$-stable prime divisors are compatibly split. 
In the next section we will prove this in a special situation. For further results on canonical splittings of $\mc G\times^{\mc P^-}Y$ we refer to \cite[Sect.~5,6]{HeTh2}.
\end{rem}

A {\it wonderful compactification} of $G/H$ is a smooth simple complete toroidal $G/H$-embedding. Recall that ``simple'' means that there is a unique closed $G$-orbit. The wonderful compactification is unique if it exists (even without the smoothness condition). Now we strengthen assumption (A4) as follows.

\smallskip
\begin{enumerate}[({\rm A}$4'$)]
\item The homogeneous space $G/H$ has a wonderful compactification $\bf Y$ for which the local structure theorem (see Section~\ref{s.prelim}) holds.
\end{enumerate}
\smallskip

For the standard properties of the wonderful compactification we refer to \cite{DeC}, \cite{Str} or \cite{BriKu}. The unique closed orbit of $\bf Y$ is naturally isomorphic to $G/Q^-$, where $Q^-$ is the opposite of the parabolic $Q$ from the local structure theorem. The cone of $\bf Y$ is the valuation cone of $G/H$ which contains no line, and $H$ contains the connected centre of $G$.

Denote the $B$-stable prime divisors of $\bf Y$ that are not $G$-stable, by ${\bf E}_1,\ldots,{\bf E}_r$ and the $G$-stable prime divisors by ${\bf Y}_1,\ldots,{\bf Y}_n$. Let $\tilde G$ be the simply connected cover of the derived group of $G$ and let $\tilde T$ and $\tilde B$ be the maximal torus and Borel subgroup of $\tilde G$ corresponding to $T$ and $B$. Note that ${\rm Pic}(G/Q^-)={\rm Pic}^{\tilde G}(G/Q^-)$ embeds in $X(\tilde T)$ as the $\chi$ with $\la\chi,\alpha^\vee\ra=0$ for all simple roots $\alpha$ that are roots of $M$. We will use the following convention: if $A$ is a subgroup of ${\rm Pic}(G/Q^-)$ we write $\chi\in A$ when we mean that $\chi\in X(\tilde T)$ occurs in the subgroup of $X(\tilde T)$ corresponding to $A$. If ${\rm Pic}({\bf Y})$ embeds in ${\rm Pic}(G/Q^-)$ (see below) and $\chi\in{\rm Pic}({\bf Y})$, then we denote the corresponding line bundle on $\bf Y$ by $\mc L_{\bf Y}(\chi)$.

We denote by $\rho_Q\in\mb Q\ot X(T)$ the half sum of the roots of $T$ in $R_uQ$, and similarly for any parabolic subgroup of $G$. Note that $\rho_B=\rho$, the half sum of the positive roots of $G$. The longest element of the Weyl group of $G$ is denoted by $w_0$. For each $i\in\{1,\ldots,r\}$ we denote by $\sigma_i$ the canonical section of $\mc L_{\bf Y}({\bf Y}_j)$.

The theorem below is essentially contained in \cite[Ch.~6]{BriKu}. Since it has nothing to do with induction, we formulate it for $\bf Y$. We only indicate what has to be added to the arguments in \cite{BriKu}.

\begin{thm}[{cf.~\cite[Ch.~6]{BriKu}}]\label{thm.wful}
Assume (A4$'$). Then the variety ${\bf Y}_0$ is an affine space and ${\rm Pic}(\bf Y)$ is freely generated by ${\bf E}_1,\ldots,{\bf E}_r$. Now assume in addition that the map ${\rm Pic}({\bf Y})\to {\rm Pic}(G/Q^-)$, given by pull-back, is injective.
\begin{enumerate}[{\rm(i)}]
\item For every $\chi\in{\rm Pic}({\bf Y})$ there exists a $\tilde B$-semi-invariant rational section $\tau_\chi$ of $\mc L_{\bf Y}(\chi)$ which has weight $\chi$ and restricts to a nonzero rational section on $G/Q^-$; it is unique up to nonzero scalar multiples.
\item ${\omega_{\bf Y}}^{-1}=\mc L_{\bf Y}(2\rho_{Q})\otimes\mc O_{\bf Y}(\sum_j{\bf Y}_j)$.
\item Assume $Q=B$, $\rho\in{\rm Pic}({\bf Y})$ and that $\tau_\rho$ is a global section. Put $\tau^-_\rho=w_0\cdot\tau_\rho$. Then $(\tau^-_\rho\tau_\rho\sigma_1\cdots\sigma_r)^{p-1}\in H^0({\bf Y},\omega_{\bf Y}^{1-p})$ is (up to a scalar multiple) a $B$-canonical Frobenius splitting of $\bf Y$.
\end{enumerate}
\end{thm}

\begin{proof}
By (A4$'$) the variety $\ov{T\cdot y}^0\subseteq\bf Y$ is a smooth affine toric variety with a $T$-fixed point, so it is an affine space. For the first statement the argument is now the same as in the proof of \cite[Lem.~6.1.9]{BriKu}.\\
(i).\ The argument is almost the same as that in the proof of \cite[Prop.~6.1.11(iv)]{BriKu}. The canonical section of $\mc O_{\bf X}({\bf E}_i)$ is $\tilde B$-semi-invariant 
and its restriction to $G/Q^-$ is nonzero, since no ${\bf E}_i$ contains a $G$-orbit. Using these canonical sections we construct the required sections $\tau_\chi$. The statement about their weight follows by considering the restriction to $G/Q^-$.\\
(ii).\ This follows, exactly as in the proof of \cite[Prop.~6.1.11(v)]{BriKu}, from the adjunction formula and the injectivity of ${\rm Pic}({\bf Y})\to {\rm Pic}(G/Q^-)$.\\
(iii).\ The arguments are standard (see \cite{Str}, \cite{Rit2} or \cite{BriKu}), we briefly sketch them. Note that $\tau^-_\rho\tau_\rho$ is the image of $f^-\otimes f^+$ under a homomorphism of $\tilde G$-modules ${\rm St}\ot{\rm St}\to H^0({\bf Y},2\rho)$, where ${\rm St}$ is the Steinberg module for $\tilde G$, i.e. the irreducible of highest weight $(p-1)\rho$. Since $Q=B$, $(\tau^-_\rho\tau_\rho|_{G/B^-})^{p-1}$ is a Frobenius splitting of $G/B^-$. Now $(\tau^-_\rho\tau_\rho\sigma_1\cdots\sigma_r)^{p-1}$ is a splitting by repeatedly applying \cite[Ex.~1.3.E.4]{BriKu} and it is canonical, since it is the image of $f^-\otimes f^+$ under a homomorphism of $\tilde G$-modules ${\rm St}\ot{\rm St}\to H^0({\bf Y},{\omega_{\bf Y}}^{-1})$.
\end{proof}

If $\mc G/\ov{\mc H}$ has a wonderful compactification $\bf X$, then we define ${\bf D}_1,\ldots,{\bf D}_r$ as the $\mc B$-stable prime divisors of $\bf X$ that intersect $\mc P^-/\mc H$. Note that ${\bf D}_i\cap{\bf Y}={\bf E}_i$. Furthermore, we define, for $\alpha$ a simple root that is not a root of $\mc K$, ${\bf D}_\alpha$ as the closure in $\bf X$ of inverse image $\mc Bs_\alpha\cdot\ov{x}$ in $\mc G/\ov{\mc H}$ of $\mc Bs_\alpha\mc P^-/\mc P^-\subseteq \mc G/\mc P^-$. Now define $\tilde{\mc G}$ and $\tilde{\mc T}$ analogously to $\tilde G$ and $\tilde T$. Note that $\tilde G$ is isomorphic to a product of simple factors of a Levi of $\tilde{\mc G}$, in particular $\tilde T$ is a subtorus of $\tilde{\mc T}$. Furthermore, ${\rm Pic}(\mc G/\mc Q^-)$ and ${\rm Pic}(\mc G/\mc P^-)$ embed in $X(\tilde{\mc T})$ as the $\chi$ with $\la\chi,\alpha^\vee\ra=0$ for all simple roots $\alpha$ that are roots of $\mc M$ respectively $\mc K$.

\begin{prop}\label{prop.wfulinduction}
Assume (A1)-(A3),(A4$'$). Then the homogeneous space $\mc G/\ov{\mc H}$ has a wonderful compactification $\bf X=\mc G\times^{\mc P^-}\bf Y$ and the local structure theorem holds for all toroidal $\mc G/\mc H$-embeddings. If $\bf Y$ is projective, then $\bf X$ is projective and a toroidal $\mc G/\mc H$-embedding $X$ is quasi-projective whenever $\ov{T\cdot x}^0\subseteq X$ is quasi-projective.

Now assume in addition that $\ov{\mc H}=\mc H$ and that the map ${\rm Pic}({\bf Y})\to {\rm Pic}(G/Q^-)$, given by pull-back, is injective.
\begin{enumerate}[{\rm(i)}]
\item The group ${\rm Pic}(\bf X)$ is freely generated by the ${\bf D}_i$ and the ${\bf D}_\alpha$, and the map ${\rm Pic}({\bf X})\to {\rm Pic}(\mc G/\mc Q^-)$ is injective.
\item Assume that no ${\bf D}_i$ contains a ${\bf D}_\alpha\cap\mc G/\mc Q^-$. Under the homomorphism that sends ${\bf E}_i$ to ${\bf D}_i$, ${\rm Pic}(\bf Y)$ embeds in ${\rm Pic}(\bf X)$ as the $\chi$ with $\la\chi,\alpha^\vee\ra=0$ for all simple roots $\alpha$ that are not roots of $\mc K$. Furthermore, if for the section $\tau_\chi$ of Theorem~\ref{thm.wful}(i) (applied to $\bf X$) we have $(\tau_\chi)|_{\bf Y}=\sum_ia_i{\bf E}_i$, then $(\tau_\chi)=\sum_ia_i{\bf D}_i+\sum_\alpha\la\chi,\alpha^\vee\ra {\bf D}_\alpha$.
\end{enumerate}
\end{prop}

\begin{proof}
By Proposition~\ref{prop.toroidal}, $\bf X$ is a wonderful compactification of $\mc G/\ov{\mc H}$ for which the local structure theorem holds and which is projective if $\bf Y$ is projective. Since every toroidal $\mc G/\mc H$-embedding dominates $\bf X$, we get, by Proposition~\ref{prop.toroidal}(i), that the local structure theorem holds for all $\mc G/\mc H$-embeddings. The assertion about quasi-projectivity follows as in \cite[Prop.~6.2.3(iv)]{BriKu} using \cite[Prop.~4.6.13(ii)]{Gro}. Now assume that $\ov{\mc H}=\mc H$ and that the map ${\rm Pic}({\bf Y})\to {\rm Pic}(G/Q^-)$, given by pull-back, is injective.\\
(i).\ The first assertion follows from Lemma~\ref{lem.inducedBstabledivisors} and Theorem~\ref{thm.wful}. The next assertion follows from a simple diagram chase, see also \cite[Prop.~4.1]{DeC}.\\
(ii).\ From the assumption it follows that the restrictions of the ${\bf D}_i$ to $\mc G/\mc Q^-$ lie in the span of the $\mc B$-stable prime divisors of $\mc G/\mc Q^-$ that are not pull-backs from $\mc G/\mc P^-$. These correspond to the $\chi\in{\rm Pic}(\mc G/\mc Q^-)$ with $\la\chi,\alpha^\vee\ra=0$ for all simple roots $\alpha$ that are not roots of $\mc K$. This implies the first assertion and the second assertion follows from this.
\end{proof}

\begin{cornn}
Assume (A1)-(A3),(A4$'$).
\begin{enumerate}[{\rm(i)}]
\item If $\bf Y$ is projective, then every $\mc G/\mc H$-embedding has a $\mc G$-equivariant resolution by a quasi-projective toroidal embedding.
\item Assume in addition that $\ov{\mc H}=\mc H$. If all toroidal $G/H$-embeddings are $B$-canonically Frobenius split, then all $\mc G/\mc H$-embeddings are $\mc B$-canonically Frobenius split.
\end{enumerate}
\end{cornn}

\begin{proof}
(i).\ This follows from Proposition~\ref{prop.wfulinduction} and the arguments in the proofs of \cite[Prop.~3]{Rit2} and \cite[Prop.~6.2.5]{BriKu}.\\
(ii).\ This follows from (i), Proposition~\ref{prop.toroidal} and its corollary, and \cite[Ex.~4.1.E.3]{BriKu}.
\end{proof}

We will need the lemma below in the next section. The proof is obvious. Note that the parabolic $\mc P^{'-}$ and the morphism $\pi':\mc P^{'-}\to G$ satisfy the assumptions of \cite[Sect.~4]{DeC}.

\begin{lem}\label{lem.otherPandH}
Let $\mc P^{'-}$ be the subgroup of $\mc P^-$ generated by $\mc L$ and $\mc B^-$ and let $\mc H'$ be the scheme theoretic intersection of $\mc H$ and $\mc P'$. Then $\mc G$, $\mc P^{'-}$, $\mc H'$, $G$, $H$ and the restriction $\pi':\mc P^{'-}\to G$ of $\pi$ satisfy (A1)-(A3) (and (A4)). If we define $\mc K',\mc L',\mc Q'$ analogous to $\mc K,\mc L,\mc Q$, then $\mc L'=\mc K'=\mc L$ and if $Q=B$, then $\mc Q'=\mc B$. If $\mc H$ is the scheme theoretic inverse image of $H$ under $\pi$, then the same holds for $\mc H'$, $H$ and $\pi'$.
\end{lem}

\section{Induction from the $G\times G$-space $G$}\label{s.Ginduction}
\subsection{The general case}\label{ss.Ginductiongeneral}
In this section we obtain results in the special case of induction from the $G\times G$-space $G$. Let $T$ be a maximal torus of $G$ and let $B$ be a Borel subgroup of $G$ containing $T$. The r\^ole of the groups $G$ and $T$ from Section~\ref{s.induction} is now played by $G\times G$ and $T\times T$, and $B^-\times B$ plays the r\^ole of $B$ and $Q$. Furthermore, we take $H$ to be the diagonal in $G\times G$ and $H_{\rm a}$ the subgroup scheme of $G\times G$ generated by the diagonal and the scheme theoretic centre of $G\times G$. We assume given a connected reductive group $\mc G$ with parabolic $\mc P^-$ and closed subgroup $\mc P^-$ and a homomorphism $\pi:\mc P^-\to G\times G$ such that with $\mc H$ the scheme theoretic inverse image of $H$ under $\pi$ assumptions (A1) and (A2) from Section~\ref{s.induction} are satisfied. We define ${\mc H}_{\rm a}$ as $\mc H\mc Z$, where $\mc Z$ is the scheme theoretic inverse image in $\mc T$ of the scheme theoretic centre of $G\times G$. The groups $\mc T$, $\mc B$, $\mc K$, $\mc L$, $\mc Q$ and $\mc M$ are defined as in Section~\ref{s.induction}. Note that $(G\times G)/H=G$ and $(G\times G)/H_{\rm a}=G_{\rm ad}$ (the adjoint group). The base point of $\mc G/\mc H$ is identified with that of $G\times G/H=G$, the unit of $G$, and it is denoted by $x$. The base point of $\mc G/{\mc H}_{\rm a}$ is identified with that of $(G\times G)/H_{\rm a}=G_{\rm ad}$, the unit of $G_{\rm ad}$, and it is denoted by $x_{\rm a}$ (we abandon the notation $\ov{\mc H}$ and $\ov{x}$ from Section~\ref{s.induction}).

We can now apply the results from Section~\ref{s.induction} to $\mc G$ and $\pi:\mc P^-\to G\times G$ as above and then to the subgroups $\mc H$ and $H$ or to $\mc H$ and $H_{\rm a}$ or to ${\mc H}_{\rm a}$ and $H_{\rm a}$. In the first and the third case we have that the first group is the scheme-theoretic inverse image under $\pi$ of the second. In particular, by \cite{Str} (or \cite[Sect.~6]{BriKu}) and Proposition~\ref{prop.wfulinduction}, $\mc G/{\mc H}_{\rm a}$ has a wonderful compactification $\bf X=\mc G\times^{\mc P^-}\bf Y$ for which the local structure theorem holds, where $\bf Y$ is the wonderful compactification of $G_{\rm ad}$. Note also that the local structure theorem holds for all toroidal $\mc G/\mc H$-embeddings and that all $\mc G/\mc H$-embeddings are $\mc B$-canonically Frobenius split by \cite[Thm.~6.2.7]{BriKu} and (ii) of the corollary to Proposition~\ref{prop.wfulinduction}. We will obtain a more precise result later in this section.

Let $(\alpha_i)_{i\in I_1}$ be the simple roots of $G$. The simple roots of $\mc L$ are denoted $\alpha_{i1}$, $\alpha_{i2}$, $i\in I_1$, corresponding to the simple roots $(-\alpha_i,0)$ and $(0,\alpha_i)$ of $G\times G$. The other simple roots of $\mc K$ are denoted $(\alpha_j)_{j\in I_2}$ and the simple roots of $\mc G$ that are not roots of $\mc K$ are denoted $(\alpha_j)_{j\in J}$. So the simple roots of $\mc G$ are $\alpha_{i1}$, $\alpha_{i2}$, $i\in I_1$, and $\alpha_j$, $j\in I_2\cup J$. A connected unipotent subgroup of $G\times G$ stable under conjugation by $T\times T$ is identified with a subgroup of $\mc G$.

If $X$ is an embedding of $\mc G/\mc H$, we denote the $\mc B$-stable prime divisor $\ov{\mc Bs_{\alpha_{i1}}\cdot x}$ ($=\ov{\mc Bs_{\alpha_{i2}}\cdot x}$) of $X$ by $D_i$, $i\in I_1$, and the $\mc B$-stable prime divisor $\ov{\mc Bs_{\alpha_j}\cdot x}$ of $X$ by $D_j$, $j\in J$. So $D_i$ corresponds to the $B^-\times B$-stable prime divisor $\ov{B^-s_{\alpha_i}B}$ of $G$ and $D_j$ is the pull-back of the $\mc B$-stable prime divisor $\ov{\mc Bs_{\alpha_j}\mc P^-/\mc P^-}$ of $\mc G/\mc P^-$. We use the same notation for an embedding of $\mc G/{\mc H}_{\rm a}$. From the context it should be clear which of the two homogeneous spaces we consider and which embedding.

Now we define $\mc H'$, $\mc P^{'-}$ and $\pi':\mc P^{'-}\to G\times G$ as at the end of Section~\ref{s.induction} and we define ${\mc H}'_{\rm a}$ analogous to ${\mc H}_{\rm a}$. We denote the base point of $\mc G/\mc H'$ by $x'$. Then we can introduce the $\mc B$-stable prime divisors that are not $\mc G$-stable of any embedding of $\mc G/\mc H'$ or $\mc G/{\mc H}'_{\rm a}$ as above. They are denoted by $D_i'$, $i\in I_1$, and ${D'}_j$, $j\in J\cup I_2$. We denote the wonderful compactification of $\mc G/{\mc H}'_{\rm a}$, which exists by Proposition~\ref{prop.wfulinduction}, by ${\bf X}'$.

In case we are dealing with $\bf X$ or $\bf X'$ we will denote the prime divisors in boldface notation. Their boundary divisors are denoted by ${\bf X}_i$ and ${\bf X}_i'$, $i\in I_1$.

We have $\mc T\cdot x=(T\times T)\cdot x=T$ and we will denote the extension of $f\in k(T)$ to a $\mc U$-invariant rational function on $\mc G/\mc H$ by $\ov{f}$. We will consider a character of $T\times T$ also as a character of $\mc T$ by means of the epimorphism $\mc T\to T\times T$. Recall Remark~\ref{rem.danger} and the text after it: the $\mc T$-weight of a character $\chi\in X(T)$, as a function on $T$ or extended to $\mc B\cdot x$, is $(-\chi,\chi)$. This is the negative of its composite with the orbit map $T\times T\to T$. By the above we have embeddings $X(T)\hookrightarrow X(T\times T)\hookrightarrow X(\mc T)$ and surjections $\mb Q\ot Y(\mc T)\twoheadrightarrow\mb Q\ot Y(T\times T)\twoheadrightarrow\mb Q\ot Y(T)$. For example, under the above embeddings of character groups, $\alpha_i$ goes to $(-\alpha_i,\alpha_i)$ and then to $\alpha_{i1}+\alpha_{i2}$. Furthermore, for $\chi\in X(T)$, $t\in\mc T$ and $z\in \mc B\cdot x$ we have $\ov{\chi}(t\cdot z)=\chi(t^{-1})\ov{\chi}(z)$ \big($=(-\chi,\chi)(t^{-1})\ov{\chi}(z)$\big). This does not contradict the usual formula $\ov{\chi}(szt)=\chi(st)\ov{\chi}(z)$ on $B^-B$, because $szt$ should be read as $(s,t^{-1})\cdot z$.

We remind the reader that, because of the inclusions $X(T)\subseteq k(T)=k(\mc G/\mc H)^{\mc U}\subseteq k(\mc G/\mc H)$, every valuation of $k(\mc G/\mc H)$ determines an element of $\Hom_{\mb Z}(X(T),\mb Q)=\mb Q\ot Y(T)$. In particular this applies to the valuation associated to a prime divisor in any $\mc G/\mc H$-embedding. Similar remarks apply to the homogeneous spaces $\mc G/{\mc H}_{\rm a}$, $\mc G/\mc H'$ and $\mc G/{\mc H}'_{\rm a}$.

In the formula in assertion (ii) below the first pairing is the pairing between characters and cocharacters of $T$ and the second one is the pairing between characters and cocharacters of $\mc T$. There $\chi$ should be read as $(-\chi,\chi)\in X(\mc T)$, the $\mc T$-weight of $\ov \chi$. The first pairing can also be written like this, by replacing it by $\la\chi,\alpha_{i1}^\vee\ra$ or $\la\chi,\alpha_{i2}^\vee\ra$. Note that $\la\chi,\alpha_j^\vee\ra=0$ for all $\chi\in X(T)$ and all $j\in I_2$, since, for $j\in I_2$, $\alpha_j^\vee$ takes its values in the kernel of $\pi$.


\begin{prop}\label{prop.valuationcone}\
\begin{enumerate}[{\rm(i)}]
\item The valuation cone of $\mc G/\mc H$ is the antidominant Weyl chamber in $\mb Q\ot Y(T)$.
\item The image of $D_i$, $i\in I_1$, in $\mb Q\ot Y(T)$ is $\alpha_i^\vee$ and the image of $D_j$, $j\in J$, is the image of $\alpha_j^\vee$. Put differently, if $\ov\chi\in k(\mc G/\mc H)^{\mc U}$ extends $\chi\in X(T)$, then
    $$(\ov\chi)=\sum_{i\in I_1}\la\chi,\alpha_i^\vee\ra D_i+\sum_{j\in J}\la\chi,\alpha_j^\vee\ra D_j\ .$$
\item The analogues of (i) and (ii) for $\mc G/{\mc H}_{\rm a}$ ($T$ replaced by $T_{\rm ad}$) hold.
\item The analogues of (i) and (ii) for $\mc G/\mc H'$ and $\mc G/{\mc H}'_{\rm a}$ hold. 
\end{enumerate}
\end{prop}

\begin{proof}
(i).\ This is immediate from Proposition~\ref{prop.induction}(ii) and the corollary to Proposition~\ref{prop.limit}. Similar for the other three homogeneous spaces.\\
(ii)-(iv). The function $\ov\chi$ on $\mc G/\mc H'$ is the pull-back of that on $\mc G/\mc H$ and similarly for $\mc G/{\mc H}'_{\rm a}$ and $\mc G/{\mc H}_{\rm a}$. So the formula in (ii) only has to be proved for $\mc G/\mc H'$ and $\mc G/{\mc H}'_{\rm a}$. Since these cases are completely analogous we only prove it for $\mc G/\mc H'$. That the coefficient of $D_i$ equals $\la\chi,\alpha_i^\vee\ra$ follows by restricting to $\mc P^-/\mc H$ and applying a standard result, see e.g. \cite[Thm.~1(iii)]{T1}. To determine the other coefficients we follow the arguments in \cite{T1}. For each root $\alpha$ of $\mc G$ choose an isomorphism $\theta_\alpha:{\mb G}_a\stackrel{\sim}{\to}\mc U_\alpha$ as in \cite{Jan} II.1.1-1.3 (there denoted by $x_\alpha$). For $a\in K^\times$ define, as in \cite{Jan}, $n_\alpha(a):=\theta_\alpha(a)\theta_{-\alpha}(-a^{-1})\theta_\alpha(a)$. Then $n_\alpha(a)=\alpha^\vee(a)n_\alpha(1)$ is a representative of $s_\alpha$ and $n_\alpha(1)\theta_\alpha(a)=\alpha^\vee(-a^{-1})\theta_\alpha(-a)\theta_{-\alpha}(a^{-1})$. Now let $j\in J\cup I_2$. We have an isomorphism
\begin{equation*}
(u,a,z)\mapsto un_{\alpha_j}(1)\theta_{\alpha_j}(a)z:\mc U_{s_{\alpha_j}}\times\mb{G}_a\times \mc T\cdot x'\to s_{\alpha_j}\mc B\cdot x'\,,
\end{equation*}
where $\mc U_{s_{\alpha_j}}$ is the product of the $\mc U_\beta$ with $\beta$ positive $\ne\alpha_j$.
Since $D_j$ intersects $s_{\alpha_j}\mc B\cdot x'$ in the part with $a$-coordinate $0$, we only have to express $\ov\chi$ in the above coordinates. We get $\ov\chi(un_{\alpha_j}(1)\theta_{\alpha_j}(a)z)=\ov\chi(u{\alpha_j}^\vee(-a^{-1})\theta_{\alpha_j}(-a)\theta_{-\alpha_j}(a^{-1})z)= \ov\chi(\alpha_j^\vee(-a^{-1})z)=(-a)^{\la\chi,\alpha_j^\vee\ra}\ov\chi(z)$, since $\ov\chi$ is $\mc B$-semi-invariant of weight $\chi=(-\chi,\chi)$ and $\mc U_{-\alpha_j}\subseteq\mc H'$. Since $\ov\chi$ is regular and nowhere zero on $\mc T\cdot x'$ it follows that the coefficient of $D_j$ is $\la\chi,\alpha_j^\vee\ra$.

\end{proof}

\begin{rems}\label{rems.divisor}
1.\ From Proposition~\ref{prop.valuationcone} one obtains the divisor of $\ov \chi$ on any embedding $X$ with boundary divisors by $X_1,\ldots,X_n$ by simply adding the sum $\sum_i\la\chi,\nu_i\ra X_i$, where $\nu_i$ is the valuation corresponding to $X_i$. From this one obtains a general description of the class group of $X$, see \cite[Sect.~5.1]{Bri2}. This description can be made explicit by determining the images of the $\nu_i$ in $\mb Q\ot Y(T)$.\\
2.\ By the theory of spherical embeddings, the image of ${\bf X}_i$, $i\in I_1$, in the valuation cone of $\mc G/{\mc H}_{\rm a}$ is $-\varpi_i^\vee$, the dual fundamental weight corresponding to $\alpha_i$.
\end{rems}

As in Section~\ref{s.induction}, we denote by $\rho_{\mc Q}\in\mb Q\ot X(\mc T)$ the half sum of the roots of $\mc T$ in $R_u\mc Q$, and similarly for any parabolic subgroup of $\mc G$. Recall that $\rho_{\mc B}=\rho$, the half sum of the positive roots of $\mc G$. We remind the reader that the canonical bundle $\omega_G$ of $G$ is $G\times G$ linearised and $G\times G$-equivariantly trivial. This can be seen as follows. We have ${\rm Pic}^{G\times G}(G)=X(H)=X(G)$. Here the character is obtained by considering the $H$-action on the fiber of the line bundle over $x$. In our case this is the adjoint action of $G$ on the top exterior power of its Lie algebra. This action is clearly trivial since there is a $T$-fixed vector. It follows that $\mc G\times^{\mc P^-}\omega_G$ is $\mc G$-equivariantly trivial. In particular, if a $\sigma$ is a nonzero rational section of $\mc G\times^{\mc P^-}\omega_G$ with $(\sigma)$ $\mc B$-fixed, then it is a $\mc B$-semi-invariant and if $\sigma$ is $\mc B$-fixed, then it is $\mc G$-fixed (this follows from the corresponding statements for rational functions).

\begin{thm}\label{thm.canonicaldivisor}\ 
\begin{enumerate}[{\rm(i)}]
\item Let $X$ be an embedding of $\mc G/\mc H$ or $\mc G/{\mc H}_{\rm a}$ and denote its boundary divisors by $X_1,\ldots,X_n$. Then a canonical divisor of $X$ is given by
$$K_X=-\sum_iX_i-2\sum_{i\in I_1}D_i-\sum_{j\in J}\la2\rho_\mc Q,\alpha_j^\vee\ra D_j\ .$$
\item The analogue for an embedding of $\mc G/\mc H'$ or $\mc G/{\mc H}'_{\rm a}$ holds ($J$ replaced by $J\cup I_2$). Here $\mc Q'=\mc B$, so in the third sum all coefficients are equal to $2$.
\end{enumerate}
\end{thm}

\begin{proof}
As in the proof of \cite[Prop.~6.2.6]{BriKu} we may assume that $X$ is smooth and toroidal. Let $\theta_{T}$ be the volume form (= nowhere zero section of canonical sheaf) $\frac{dt_1\wedge\cdots\wedge dt_m}{t_1\cdots t_m}$ of $T$ and let $\theta_U$, $\theta_{U^-}$ and $\theta_{R_u\mc P}$ be the volume forms of $U$, $U^-$ and $R_u\mc P$ respectively. We will compute the divisor $(\sigma)$ of the rational section $\sigma=\theta_{R_u\mc P}\wedge\theta_{U^-}\wedge\theta_T\wedge\theta_U$ of $\varpi_X$. Here the products have to be interpreted using the isomorphism $R_u\mc P\times U^-\times T\times U\cong \mc B\cdot x$ given by the group multiplication (in the given order) and the action. By the local structure theorem for $X$ and the theory of toric varieties, the coefficients of the $X_i$ in $(\sigma)$ are all $-1$.

To compute the other coefficients we may restrict to $\mc G/\mc H$. We have an isomorphism $\varpi_{\mc G/\mc H}\cong pr^*\varpi_{\mc G/\mc P^-}\ot\varpi_{\mc G/\mc H:\mc G/\mc P^-}$. Under this isomorphism we have $\sigma=\theta_{R_u\mc P}\ot\eta$, where $\eta=\theta_{U^-}\wedge\theta_T\wedge\theta_U$. We have $\varpi_{\mc G/\mc H:\mc G/\mc P^-}\cong \mc G\times^{\mc P^-}\omega_G$ which is $\mc G$-equivariantly trivial.
Now put $\eta_1=\theta_T\wedge\theta_{U^-}\wedge\theta_U$ where the product is to be interpreted using the isomorphism $R_u\mc P\times T\times U^-\times U\cong \mc B\cdot x$ (so this section is independent of the $R_u\mc P$-coordinate). Then $\eta_1$ is regular and nowhere zero on $\mc B\cdot x$ and it is $\mc T$-fixed, so it is $\mc B$-fixed and therefore $\mc G$-fixed and nowhere zero. Now $\eta=\pm\ov{-2\rho_B}\,\eta_1$, where $\rho_B$ is the half sum of the positive roots of $G$. 
Furthermore, $(\theta_{R_u\mc P})=-\sum_{j\in J}\la2\rho_\mc P,\alpha_j^\vee\ra D_j$. So the assertion in (i) follows from Proposition~\ref{prop.valuationcone} and the fact that $\rho_\mc Q-\rho_\mc P=(-\rho_B,\rho_B)=\rho_B\in\mb Q\ot X(\mc T)$, the half sum of the positive roots from $\mc L$. Assertion (ii) follows, because $\mc G,\mc H',G$ and $\mc G,{\mc H}_{\rm a}',G_{\rm ad}$ also satisfy the hypotheses made at the beginning of this section.
\end{proof}

\begin{rem}
If $X$ has no boundary divisors, then ${\rm Cl}(X)={\rm Cl}(\mc G/\mc H)$. In this case the above proof shows that the canonical sheaf of $X$ is $\mc G$-equivariantly isomorphic to the pull-back of the canonical sheaf of $\mc G/\mc P^-$.
\end{rem}

Let $\tilde G$ be the simply connected cover of the derived group of $G$ and let $\tilde T$ and $\tilde B$ be the maximal torus and Borel subgroup of $\tilde G$ corresponding to $T$ and $B$. Define $\tilde{\mc G}$, $\tilde{\mc T}$ and $\tilde{\mc B}$ similarly. 
Recall from Section~\ref{s.induction} that ${\rm Pic}(\mc G/\mc Q^-)={\rm Pic}^{\tilde G}(\mc G/\mc Q^-)$ embeds in $X(\tilde{\mc T})$. Note that the composite $\tilde{\mc T}\to T\times T\to T_{\rm ad}\times T_{\rm ad}$ is surjective, so similar remarks apply as in the case of the surjection $\mc T\to T\times T$. Furthermore, we embed again $X(T_{\rm ad})$ in $X(T_{\rm ad}\times T_{\rm ad})\subseteq X(\tilde{\mc T})$ by $\chi\mapsto(-\chi,\chi)$.

\begin{prop}\label{prop.wfulpicard}\
\begin{enumerate}[{\rm(i)}]
\item The images of ${\bf X}_i$, ${\bf D}_i$, $i\in I_1$, and ${\bf D}_j$, $j\in J$, in ${\rm Pic}({\bf X})$ are $\alpha_{i1}+\alpha_{i2}$, $\varpi_{i1}+\varpi_{i2}$ and $\varpi_j$. The group ${\rm Pic}({\bf X})$ embeds in ${\rm Pic}(\mc G/\mc Q^-)$ as the $\chi$ with $\la\chi,\alpha_{i1}^\vee\ra=\la\chi,\alpha_{i2}^\vee\ra$ for all $i\in I_1$.
\item For the section $\tau_\chi$, $\chi\in{\rm Pic}({\bf X})$, of Theorem~\ref{thm.wful} we have
$$(\tau_\chi)=\sum_{i\in I_1}\la\chi,\alpha_{i1}^\vee\ra{\bf D}_i+\sum_{j\in J}\la\chi,\alpha_j^\vee\ra {\bf D}_j\ .$$
\item The analogues of (i) and (ii) for ${\bf X}'$ ($J$ replaced by $J\cup I_2$) hold.
\item The wonderful compactifications ${\bf X}'$ and $\bf X$ are $\mc B$-canonically Frobenius split compatible with the $\mc G$-orbit closures and the $\mc B$-stable prime divisors and $\mc B^-$-stable prime divisors that are not $\mc G$-stable. The splitting of ${\bf X}'$ is given by Theorem~\ref{thm.wful}(iii).
\end{enumerate}
\end{prop}

\begin{proof}
(i).\ Note that, by Proposition~\ref{prop.wfulinduction} ${\rm Pic}({\bf X})$ embeds in ${\rm Pic}(\mc G/\mc Q^-)$ and therefore in $X(\tilde{\mc T})$. To determine the image of ${\bf X}_i$ we argue as in \cite[Cor.~8.2]{DeCP}. The canonical section $s_i$ of $\mc L_{\bf X}({\bf X}_i)$ is $\tilde{\mc G}$-fixed. Since the restriction of $\mc L_{\bf X}({\bf X}_i)$ to ${\bf X}_0$ is trivial and $\ov{-\alpha_i}$ is an equation for ${\bf X}_i$ on ${\bf X}_0$ (see Remark~\ref{rem.wful}), $u_i:=s_i/\,\ov{-\alpha_i}$ is a rational section of $\mc L_{\bf X}({\bf X}_i)$ which is nowhere zero on ${\bf X}_0$. In particular, it restricts to a nonzero rational section on $\mc G/\mc Q^-$. Since $\ov{-\alpha_i}$ is $\tilde{\mc B}$-semi-invariant with weight $(\alpha_i,-\alpha_i)$, it follows that $u_i$ is $\tilde{\mc B}$-semi-invariant with weight $(-\alpha_i,\alpha_i)=\alpha_{i1}+\alpha_{i2}$. The image of ${\bf D}_j$, $j\in J$, is $\varpi_j$, since it is the pull-back of the corresponding prime divisor on $\mc G/\mc P^-$.

Now let $i\in I_1$ and choose $\mu\in X(T_{\rm ad})$ with $\la\mu,\alpha_i^\vee\ra=l>0$ and $\la\mu,\alpha_{i'}^\vee\ra=0$ for all $i'\in I_1\sm\{i\}$. Then $\mu=(-\mu,\mu)=l(-\varpi_i,\varpi_i)+\sum_{j\in J}\la\mu,\alpha_j^\vee\ra\varpi_j$. Put ${\bf E}=l{\bf D}_i+\sum_{j\in J}\la\mu,\alpha_j^\vee\ra{\bf D}_j$. Then $(\ov\mu)=\bf E$ on $\mc G/{\mc H}_{\rm a}$ by Proposition~\ref{prop.valuationcone}(iii), so the restriction of $\mc L_{\bf X}({\bf E})$ to $\mc G/{\mc H}_{\rm a}$ is ($\tilde G$-equivariantly) trivial. Now let $\tau_i$ be the canonical section of ${\bf D}_i$ and for each $j\in J$, let $\tau_j$ be the canonical section of ${\bf D}_j$. Then the restriction of the section $\tau_i^l\prod_{j\in J}\tau_j^{\la\mu,\alpha_j^\vee\ra}$ of $\mc L_{\bf X}({\bf E})$ to $\mc G/{\mc H}_{\rm a}$ must be a scalar multiple of $(\ov\mu)$, since, by Proposition~\ref{prop.valuationcone}(iii), a $\tilde{\mc B}$-semi-invariant function on $\mc G/{\mc H}_{\rm a}$ is determined up to a scalar multiple by its divisor. So $\tau_i$ has weight $(-\varpi_i,\varpi_i)=\varpi_{i1}+\varpi_{i2}$. The final assertion is now clear.\\
(ii).\ By (i), the formula is true for $\chi=\varpi_{i1}+\varpi_{i2}$, $i\in I_1$, and $\chi=\varpi_j$, $j\in J$. But then it is true for all $\chi\in{\rm Pic}({\bf X})$, since $\chi\mapsto(\tau_\chi)$ is a homomorphism.\\
(iii).\ This follows, because $\mc G,{\mc H}_{\rm a}',G_{\rm ad}$ also satisfy the hypotheses made at the beginning of this section.\\
(iv).\ Note that $\mc Q'=\mc B$ and that $\rho\in{\rm Pic}({\bf X}')$. So, by Theorem~\ref{thm.wful}, the $(p-1)$-th power of the global section $\tau_{\rho}^{'-}\tau_{\rho}'\prod_{i\in I_1}\sigma'_i$ of $\omega_{{\bf X}'}$ is a $\mc B$-canonical Frobenius splitting of ${\bf X}'$. Recall that $\tau_{\rho}^{'-}=w_0\cdot\tau'_{\rho}$, where $w_0$ is the longest element of the Weyl group of $\mc G$. By (iii) its divisor is the union of the ${\bf X}'_i$, the ${\bf D}'_i$ and $w_0{\bf D}'_i$, $i\in I_1$, and the ${\bf D}'_j$ and $w_0{\bf D}'_j$, $j\in J\cup I_2$. So the statement about the compatible splitting follows from the second assertion of \cite[Prop.~1.3.11]{BriKu}, \cite[Prop.~1.2.1]{BriKu} and the fact that the $\mc G$-orbit closures are transversal intersections of the ${\bf X}'_i$. The assertion about $\bf X$ follows by considering the morphism ${\bf X}'=\mc G\times^{\mc P^{'-}}{\bf Y}\to\mc G\times^{\mc P^-}{\bf Y}={\bf X}$ and applying Lemma~\ref{lem.flagbundle} and \cite[Ex.~4.1.E.3]{BriKu}.
\end{proof}

A {\it rational resolution} is a resolution $\varphi:\tilde X\to X$ which is a rational morphism (see the definition before Lemma~\ref{lem.vanishing}) and satisfies $R^i\varphi_*\omega_X=0$ for all $i>0$. The argument in the proof of (ii) below was also used in \cite{MT1} and \cite{MT2}.

\begin{thm}\label{thm.Frobsplratres}\ 
\begin{enumerate}[{\rm(i)}]
\item All embeddings of  $\mc G/\mc H$ and $\mc G'/\mc H'$ are $\mc B$-canonically Frobenius split compatible with the $\mc G$-orbit closures and the $\mc B$-stable prime divisors and $\mc B^-$-stable prime divisors that are not $\mc G$-stable. If $X'$ is a smooth complete toroidal $\mc G'/\mc H'$-embedding with boundary divisors $X'_i$, then the $\mc G$-linearised line bundle $\omega_{X'}^{-1}\ot\mc L_{X'}(-\sum_iX'_i)$ is the pull-back of $\mc L_{{\bf X}'}(2\rho)$ and the splitting is given by the $(p-1)$-th power of the global section of $\omega_{X'}^{-1}$ which is the pull-back of $\tau^{'-}_\rho\tau'_\rho$ multiplied with the canonical sections of the $\mc L_{X'}(X'_i)$.
\item Let $X$ be a $\mc G/\mc H$-embedding and let $\varphi:\tilde X\to X$ be a resolution as in (i) of the corollary to Proposition~\ref{prop.wfulinduction}, then $\varphi$ is a rational resolution.
\end{enumerate}
\end{thm}

\begin{proof}
(i).\ By \cite[Ex.~4.1.E.3]{BriKu} and (i) of the corollary to Proposition~\ref{prop.wfulinduction}, we may assume that the embedding is smooth complete (projective, in fact) and toroidal. First we consider such an embedding $X$ of $\mc G/\mc H$. By Proposition~\ref{prop.toroidal}(ii) $X=\mc G\times^{\mc P^-}Y$, where $Y$ is a smooth complete toroidal $G$-embedding. Then $X'=\mc G\times^{\mc P^{'-}}Y$ is a smooth complete toroidal $\mc G/\mc H'$-embedding. By \cite[Ex.~4.1.E.3]{BriKu} applied to the morphism $X'\to X$ we are now reduced to $\mc G/\mc H'$-embeddings.

So let $X'$ be a smooth complete toroidal $\mc G/\mc H'$-embedding with boundary divisors $X'_i$. We give each $\mc L_{X'}(X'_i)$ the $\mc G$-linearisation for which the canonical section is $\mc G$-fixed. If we multiply the section from the proof of Theorem~\ref{thm.canonicaldivisor} with the canonical sections of the $\mc L_{X'}(X'_i)$, then we obtain a $\mc B$-semi-invariant rational section of $\omega_{X'}\ot\mc L_{X'}(\sum_iX'_i)$ which has weight $-2\rho$ and divisor minus twice the sum of the $\mc B$-stable prime divisors. By Proposition~\ref{prop.wfulpicard} and Proposition~\ref{prop.induction}(iv) the pull-back of $\tau'_{-2\rho}$ is a section of the pull-back of $\mc L_{{\bf X}'}(-2\rho)$ with the same properties. On a complete variety a rational section is determined up to scalar multiples by its divisor. Furthermore, any two $\mc G$-linearisations differ by a character. It follows that $\omega_{X'}\ot\mc L_{X'}(\sum_iX'_i)$ is $\mc G$-equivariantly isomorphic to the pull-back of $\mc L_{{\bf X}'}(-2\rho)$.
The rest of the proof is the same as \cite[Thm.~2]{Rit2}, see also \cite[Thm.~6.2.7]{BriKu}.\\
(ii).\
By (i) and Lemma~\ref{lem.vanishing} it suffices to that $R^i\varphi_*\omega_{\tilde X}=0$ for all $i>0$. Since $\tilde X$ is toroidal, we have by Proposition~\ref{prop.toroidal}(ii) that $\tilde X=\mc G\times^{\mc P^-}\tilde Y$, where $\tilde Y$ is a toroidal $G$-embedding. Put $\tilde X'=\mc G\times^{\mc P^{'-}}\tilde Y$ and let $\psi:\tilde X'\to\tilde X$ be the canonical morphism. By Lemma's~\ref{lem.flagbundle} and~\ref{lem.duality}, the Grothendieck spectral sequence for $\varphi_*$, $\psi_*$ and $\omega_{\tilde X'}$ collapses and we obtain $(R^i\varphi_*)(\omega_{\tilde X})=(R^i\varphi_*)(R^d\psi_*\omega_{\tilde X'})=R^{i+d}(\varphi\circ\psi)_*(\omega_{\tilde X'})$. So it suffices to show that $R^{i+d}(\varphi\circ\psi)_*(\omega_{\tilde X'})$ vanishes for $i>0$. By the above equality and since $\varphi$ is the identity on the open $\mc G$-orbit, this sheaf vanishes on the complement $E$ of the open $\mc G$-orbit in $X$. Now the result follows from (i)
and the Mehta-Van der Kallen Theorem (see \cite{MvdK} or \cite[Thm.~1.2.12]{BriKu}) applied with $D$ the anti-canonical divisor of $\tilde X'$, i.e. the union of the boundary divisors and the $\mc B$-stable prime divisors and $\mc B^-$-stable prime divisors that are not $\mc G$-stable.
\end{proof}

\subsection{Reductive monoids}\label{ss.redmon}
Throughout this section $M$ denotes a normal reductive monoid with unit group $G$. Let $T$ be a maximal torus of $G$, let $B$ be a Borel subgroup of $G$ containing $T$ and let $U$ be its unipotent radical. The opposite Borel subgroup and its unipotent radical are denoted by $B^-$ and $U^-$. For each root $\alpha$ we denote the corresponding root subgroup by $U_\alpha$. As is well-known, any $G\times G$-orbit in $M$ is of the form $GeG$ for some idempotent $e\in\ov T$, where $\ov T$ is the closure of $T$ in $M$. See e.g. \cite[Lem.~3]{Rit1}.

From now on $e$ denotes an idempotent in $\ov T$. As in \cite[Sect.~4.2]{Ren2} we define $P=\{x\in G\,|\,xe=exe\}$ and $P^-=\{x\in G\,|\,ex=exe\}$. These are opposite parabolic subgroups with common Levi subgroup $L=C_G(e)$ and with $eR_uP^-=(R_uP)e=\{e\}$. Note that $C_{G\times G}(e)\subseteq P\times P^-$. The map $x\to ex:L\to Le$ is a homomorphism of reductive groups. Let $L_1$ be the subgroup of $L$ generated by $T$ and the $U_\alpha$, $\alpha$ a root of $L$, which don't lie in the kernel of this homomorphism and let $L_2$ be the subgroup of $L$ generated by $T$ and the other $U_\alpha$. Note that these are Levi subgroups of parabolics in $L$ (or in $G$). We denote $B\cap L$ by $B_L$ and similarly for $U$ and for $L_1$ and $L_2$. The argument in the proof of assertion (ii) below was mentioned to me by L.~Renner.

\begin{prop}[cf.~\cite{Ren1}]\ \label{prop.separable}
\begin{enumerate}[{\rm(i)}]
\item The morphism $T\to Te$ is separable and its kernel is connected.
\item Put $N=\{x\in\ov L_2\,|\, ex\in eT\}$. Then multiplication defines open embeddings
\begin{align*}
&R_uP^-\times U^-_{L_1}\times N\times U_{L_1}\times R_uP\to M,\\
&R_uP^-\times U^-_{L_1}\times Te\times U_{L_1}\times R_uP\to GeG,\\
&U^-_{L_1}\times N\times U_{L_1}\to\ov{L},\\
&U^-_{L_1}\times Te\times U_{L_1}\to Le.
\end{align*}
\item The morphisms $L\to Le$ and $\ L\times L\to Le$ and $G\times G\to GeG$ are all separable.
\end{enumerate}
\end{prop}

\begin{proof}
(i).\ Since the closure of $T$ is normal by \cite[Cor.~6.2.14]{BriKu} or \cite[Ex.~4.6.2.8]{Ren2}, this follows from the theory of toric varieties (\cite[Ch.~1]{KKMS}). In fact the morphism $T\to Te$ corresponds to the embedding of the characters of $T$ that vanish on the face corresponding to $e$ into the full character group of $T$.\\
(ii).\ We only have to prove that the first map is an open embedding. Since $M$ is normal, it suffices, by Zariski's Main Theorem, to show that this map is injective. Assume that $u_1x_1v_1=u_2x_2v_2$, where $u_i\in R_uP^-U^-_{L_1}$, $x_i\in N$ and $v_i\in U_{L_1}R_uP$.
Then $$u_1x_1ev_1e=u_1x_1v_1e=u_2x_2v_2e=u_2x_2ev_2e.$$
So, by the definition of $N$, there exist $t_1,t_2\in T$ such that
$$u_1t_1ev_1e=u_2t_2ev_2e.$$
Multiplying on the left by $t_1^{-1}u_2^{-1}$ and on the right with $v_1^{-1}e$ we get
$$t_1^{-1}u_2^{-1}u_1t_1e=t_1^{-1}t_2ev_2v_1^{-1}e.$$
Let $\lambda$ be a cocharacter of $T$ that defines $P$. Then the limit along $\lambda$ (for the action by conjugation) of the RHS exists and the limit along $-\lambda$ of the LHS exists. So, since $M$ is affine, both sides are centralised by $\lambda$. Taking limits along $\lambda$ in the above equality, we obtain $t_1^{-1}u_2^{-1}u_1t_1e=t_1^{-1}t_2e$. Taking limits along $-\lambda$ in this equality, we obtain $e=t_1^{-1}t_2e$. So $u_1=u_2$ by \cite[Lem.~6.2]{Ren1}. Similarly, we obtain $v_1=v_2$ and therefore also $x_1=x_2$.\\
(iii).\ By (i) and the last open embedding in (ii) we have that the morphism $L\to Le$ is separable. Then the same must hold for the morphism $(g,h)\to geh^{-1}=gh^{-1}e:L\times L\to Le$, since it is the composite of two separable morphisms. Then the same must hold for the morphism $P\times P^-\to Le$, again since it is the composite of two separable morphisms. But the scheme theoretic stabiliser of $e$ in $G\times G$ is contained in $P\times P^-$, so it is reduced and the orbit map $G\times G\to GeG$ is separable.
\end{proof}

Define $\mc G=G\times G$, $\mc T=T\times T$, $\mc B=B^-\times B$, $\mc P^-=P\times P^-$, $\mc H=C_{G\times G}(e)$ and let $\pi:\mc P^-\to Le\times Le$ be the canonical morphism. By Proposition~\ref{prop.separable} we can now apply the results from Section~\ref{ss.Ginductiongeneral} with $G$ replaced by $Le$. Then $H$ is the diagonal in $Le\times Le$, $H_{\rm a}$ the subgroup scheme of $Le\times Le$ generated by the diagonal and the scheme theoretic centre and ${\mc H}_{\rm a}=C_{G\times G}(e)\mc Z$, where $\mc Z$ is the scheme theoretic inverse image in $T\times T$ of the scheme theoretic centre of $Le\times Le$. The group $B$ from Section~\ref{ss.Ginductiongeneral} is now equal to $B_Le=B_{L_1}e$. Note furthermore that $\mc K=L\times L$, $\mc L=L_1\times L_1$, $\mc M=L_2\times L_2$ and $\mc Q=Q^-\times Q$, where $Q=R_uPU_{L_1}L_2$.

Let $\alpha_1,\ldots,\alpha_r$ be the simple roots of $G$ and partition the set $\{1,\ldots,r\}$ into three subsets $I_1$, $I_2$ and $J$ such that the $\alpha_i$ with $i\in I_j$ are the simple roots of $L_j$ and such that the $\alpha_i$ with $i\in J$ are the simple roots that are not roots of $L$. The sets $I_2$ and $J$ are now different from Section~\ref{ss.Ginductiongeneral}. The roots $\alpha_j$, $j\in J\cup I_2$, from Section~\ref{ss.Ginductiongeneral} now come in pairs and are denoted $\alpha_{j1}=(-\alpha_j,0)$, $\alpha_{j2}=(0,\alpha_j)$, $j\in I_2\cup J$. We denote the corresponding pull-back divisors by $D_j^1$ and $D_j^2$. Proposition~\ref{prop.valuationcone} and Theorem~\ref{thm.canonicaldivisor} can now be expressed in the new notation by replacing in the second sum in the formulas $D_j$ by $D_j^1+D_j^2$ and $\rho_{\mc Q}$ by $\rho_Q$. The formula in Proposition~\ref{prop.wfulpicard} now becomes
$$(\tau_\chi)=\sum_{i\in I_1}\la\chi_2,\alpha_i^\vee\ra{\bf D}_i+\sum_{j\in J}\big(-\la\chi_1,\alpha_j^\vee\ra {\bf D}^1_j+\la\chi_2,\alpha_j^\vee\ra {\bf D}^2_j\,\big)\ ,$$
for $\chi=(\chi_1,\chi_2)\in{\rm Pic}({\bf X})\subseteq X(\tilde T)\times X(\tilde T)$.

\section{Further examples}\label{s.examples}
In this section we briefly indicate what the different groups from Sections~\ref{s.induction} and \ref{s.Ginduction} are in the case of determinantal varieties and varieties of (circular) complexes. After that we determine the class groups of a certain reductive monoid and the varieties of circular complexes and determine when these varieties are Gorenstein. These questions were answered by Svanes, Bruns and Yoshino for the determinantal varieties and the varieties of complexes. See \cite{Bru}, \cite{Y} and the references there. See also \cite{DeC}, \cite{Bri1}, \cite{MT1} and \cite{MT2}.

Throughout this section we will be dealing with normal cones, i.e. with normal closed subvarieties of some affine space which are stable under scalar multiplication. Such varieties have trivial Picard group. See e.g. \cite[Lem.~2.1]{Bri1} (one needs that $L$ is linearly reductive, and in our case $L=k^\times$, the multiplicative group of the field).
The normality of the varieties in this section is well-known. One can show it for example by first showing that the coordinate ring has a good filtration (e.g. by tableaux methods as in \cite{DeCS} or \cite{T2}) and then applying the method of $U$-invariants (see \cite[Thm.~17]{Gr1}).
\subsection{Determinantal varieties and varieties of (circular) complexes}\label{ss.detcirc}\ \\
{\bf The determinantal varieties}.
Let $m,n,r$ be integers with $0\le r\le m,n$ and let $X_r$ be the variety of $m\times n$ matrices of rank $\le r$. Put $\mc G=\GL_m\times\GL_n$ and put $G=\GL_r$. Let $\mc B$ be the Borel subgroup group that consists of the pairs of matrices $(A,B)$ with $A$ lower triangular and $B$ upper triangular. The open $\mc B$-orbit in $X_r$ is the $\mc B$-orbit of
\begin{equation*}
E_r=\begin{bmatrix}
I_r&0\\
0&0
\end{bmatrix},
\end{equation*}
where $I_r$ is the $r\times r$ identity matrix.
The stabiliser $\mc H$ of $E_r$ in $\mc G$ consists of the pairs of matrices
\begin{equation*}
(A,B)=\Big(
\begin{bmatrix}
A_{11}&A_{12}\\
0&A_{22}
\end{bmatrix},
\begin{bmatrix}
B_{11}&0\\
B_{21}&B_{22}
\end{bmatrix}
\Big)\,,
\end{equation*}
with $A_{11}=B_{11}$. If we drop this condition, then we obtain the parabolic $\mc P^-$, and it will be clear what the Levi $\mc K$ and the opposite parabolic $\mc P$ are. The subgroup $\mc L$ of $\mc P^-$ consist of the pairs $(A,B)\in\mc P^-$ with $A_{12}=0$, $B_{21}=0$ and $A_{22}$ and $B_{22}$ diagonal. The subgroup $\mc Q$ of $\mc P$ consist of the pairs $(A,B)\in\mc P$ with $A_{11}$ lower triangular and $B_{11}$ upper triangular. Its Levi $\mc M$ consists of the pairs $(A,B)\in\mc P$ with $A_{21}=0$, $B_{12}=0$ and $A_{11}$ and $B_{11}$ diagonal. The homomorphism $\mc P^-\to G\times G$ is given by $(A,B)\mapsto (A_{11},B_{11})$. Finally, the group $\mc P^{'-}$ consists of the pairs $(A,B)\in\mc P^{'-}$ with $A_{22}$ upper triangular and $B_{22}$ lower triangular.\\
{\bf The varieties of complexes}. Let $l,m,n,r,s$ be integers with $0\le r\le l$, $0\le s\le n$ and $r+s\le m$ and let $X_{rs}$ be the variety of pairs of matrices $(A,B)\in\Mat_{l,m}\times\Mat_{m,n}$ with $\rk\,A\le r$, $\rk\,B\le s$ and $AB=0$.
Put $\mc G=\GL_l\times\GL_m\times\GL_n$ and put $G=\GL_r\times\GL_s$. The action of $\mc G$ on $X_{rs}$ is given by $(g_1,g_2,g_3)\cdot (A,B)=(g_1Ag_2^{-1},g_2Bg_3^{-1})$. Let $\mc B$ be the Borel subgroup group that consists of the triples of matrices $(A,B,C)$ with $A$ and $C$ lower triangular and $B$ upper triangular. The open $\mc B$-orbit in $X$ is the $\mc B$-orbit of
\begin{equation*}
E_{rs}=(E_r,F_s)=\Big(\begin{bmatrix}
I_r&0&0\\
0&0&0
\end{bmatrix},
\begin{bmatrix}
0&0\\
0&0\\
0&I_s
\end{bmatrix}\Big),
\end{equation*}
where $I_r$ is the $r\times r$ identity matrix and $I_s$ is the $s\times s$ identity matrix.
The stabiliser $\mc H$ of $E_{rs}$ in $\mc G$ consists of the triples of matrices
\begin{equation*}
(A,B,C)=\Big(
\begin{bmatrix}
A_{11}&A_{12}\\
0&A_{22}
\end{bmatrix},
\begin{bmatrix}
B_{11}&0&0\\
B_{21}&B_{22}&0\\
B_{31}&B_{32}&B_{33}
\end{bmatrix},
\begin{bmatrix}
C_{11}&C_{12}\\
0&C_{22}
\end{bmatrix}
\Big)\,,
\end{equation*}
with $A_{11}=B_{11}$ and $B_{33}=C_{22}$. If we drop these conditions, then we obtain the parabolic $\mc P^-$. The homomorphism $\mc P^-\to G\times G$ is given by $(A,B,C)\mapsto\big((A_{11},B_{33}),(B_{11},C_{22})\big)$.\\
{\bf The varieties of circular complexes}. Let $m,n,r,s$ be integers $\ge0$ with $r+s\le m,n$ and let $X_{rs}$ be the variety of pairs of matrices $(A,B)\in\Mat_{m,n}\times\Mat_{n,m}$ with $\rk\,A\le r$, $\rk\,B\le s$, $AB=0$ and $BA=0$.
Put $\mc G=\GL_m\times\GL_n$ and put $G=\GL_r\times\GL_s$. The action of $\mc G$ on $X_{rs}$ is given by $(g_1,g_2)\cdot (A,B)=(g_1Ag_2^{-1},g_2Bg_1^{-1})$. Let $\mc B$ be the Borel subgroup group that consists of the pairs of matrices $(A,B)$ with $A$ lower triangular and $B$ upper triangular. The open $\mc B$-orbit in $X_{rs}$ is the $\mc B$-orbit of
\begin{equation*}
E_{rs}=(E_r,F_s)=\Big(\begin{bmatrix}
I_r&0&0\\
0&0&0\\
0&0&0
\end{bmatrix},
\begin{bmatrix}
0&0&0\\
0&0&0\\
0&0&I_s
\end{bmatrix}\Big),
\end{equation*}
where $I_r$ is the $r\times r$ identity matrix and $I_s$ is the $s\times s$ identity matrix.
The stabiliser $\mc H$ of $E_{rs}$ in $\mc G$ consists of the pairs of matrices
\begin{equation*}
(A,B)=\Big(
\begin{bmatrix}
A_{11}&A_{12}&A_{13}\\
0&A_{22}&A_{23}\\
0&0&A_{33}
\end{bmatrix},
\begin{bmatrix}
B_{11}&0&0\\
B_{21}&B_{22}&0\\
B_{31}&B_{32}&B_{33}
\end{bmatrix}
\Big)\,,
\end{equation*}
with $A_{11}=B_{11}$ and $A_{33}=B_{33}$. If we drop these conditions, then we obtain the parabolic $\mc P^-$. The homomorphism $\mc P^-\to G\times G$ is given by $(A,B)\mapsto\big((A_{11},B_{33}),(B_{11},A_{33})\big)$.\\
\subsection{The class group and the Gorenstein property}\ \\
{\bf The monoid $M$}.
Let $m$ be an integer $\ge 1$ and put $n=2m$. As in \cite[Sect.~3]{T2} we define $M$ as the monoid of $m\times m$ matrices
$\begin{bmatrix}
A&0\\
0&B
\end{bmatrix}$
with $A,B\in\Mat_m$ and $A^TB=AB^T=d(A,B)I$ for some scalar $d(A,B)\in k$. Here $A^T$ denotes the transpose of a matrix $A$ and $I$ is the $m\times m$ identity matrix. Then $M$ is a connected reductive closed submonoid of $\Mat_m\times\Mat_m\subseteq\Mat_n$ and $d$ is a regular function on $M$ which we call the {\it coefficient of dilation}. We denote the group of invertible elements of $M$ by $G$. Note that $M\sm G$ is the zero locus of $d$. Let $T$ be the maximal torus of $G$ that consists of diagonal matrices. We denote the Borel subgroup of $G$ which consists of the $(A,C)\in G$ with $A$ upper triangular (then the matrix $C$ is lower triangular) by $B$. Denote the restrictions to $T$ of the standard coordinates on the diagonal matrices by $\ve_1,\ldots,\ve_n$. Then we have relations $\ve_i+\ve_{m+i}=\ve_j+\ve_{m+j}$ for all $i,j\in\{1,\ldots,m\}$. The simple roots are $\alpha_i=\ve_i-\ve_{i+1}$, $i\in\{1,\ldots,m-1\}$. We denote the corresponding $B^-\times B$-stable prime divisors of $X$ by $D_1,\ldots,D_{m-1}$.

For $r,s\ge0$ with $r+s\le m$ we define $X_{rs}$ as the subvariety of $M$ which consists of the $(A,B)\in M\sm G$ with $\rk\,A\le r$ and $\rk\,B\le s$. This subvariety is $G\times G$-stable and irreducible. For $r\in\{0,\ldots,m\}$ put $X_r=X_{r,m-r}$. Then the $G\times G$-stable prime divisors of $M$ are $X_0,\ldots, X_m$. The open $B^-\times B$-orbit in $X_{rs}$ is the $B^-\times B$-orbit of $E_{r,s}=(E_r,F_s)$, where $E_r,F_{s}\in\Mat_m$ are defined by
\begin{equation*}
E_r=\begin{bmatrix}
I_r&0\\
0&0
\end{bmatrix},
F_s=\begin{bmatrix}
0&0\\
0&I_s
\end{bmatrix},
\end{equation*}
where $I_r$ and $I_s$ are the obvious identity matrices. Now define $\lambda_r\in Y(T)$ by
$$\lambda_r(t)=E_{r,m-r}+t(I_n-E_{r,m-r}).$$
Then $\lambda_r$ is indivisible and its limit exists and lies in the open $B^-\times B$-orbit of $X_r$. So, by the corollary to Proposition~\ref{prop.limit}, the image of (the valuation corresponding to) $X_r$ in $\mb Q\ot Y(T)$ is $\lambda_r$.

\begin{prop}\
\begin{enumerate}[{\rm(i)}]
\item The class group of $M$ is freely generated by the images of $D_1,\ldots,D_{m-1}$.
\item A canonical divisor of $M$ is given by $-2\sum_{i=1}^{m-1}D_i$. So if $m>1$, then $M$ is not Gorenstein.
\end{enumerate}
\end{prop}

\begin{proof}
(i).\ Clearly, $X(T)$ is generated by $\ve_1,\ldots,\ve_{m+1}$. By \cite[Sect.~5.1]{Bri2} ${\rm Cl}(M)$ is the quotient of the free group on the $D_i$ and $X_i$ by the subgroup whose elements are the $(\ov\chi)$, $\chi\in X(T)$. By Proposition~\ref{prop.valuationcone} we have $(\ov{\ve_1})=X_0+D_1$, $(\ov{\ve_i})=\sum_{j=0}^{i-1}X_j+D_i-D_{i-1}$ for $i\in\{2,\ldots,m-1\}$, $(\ov{\ve_m})=\sum_{j=0}^{m-1}X_j-D_{m-1}$ and $(\ov{\ve_{m+1}})=\sum_{j=1}^mX_j-D_1$. The assertion follows.\\
(ii).\ Since $d=\ov{\ve_1+\ve_{m+1}}$, we have, by Proposition~\ref{prop.valuationcone}, $(d)=\sum_{i=0}^mX_i$. So the assertion follows from Theorem~\ref{thm.canonicaldivisor} (or \cite[Prop.~4]{Rit2} or \cite[Prop.~6.2.6]{BriKu}).
\end{proof}

\noindent{\bf Circular complexes}. 
Let $m,n,r,s$ be integers $\ge0$ with $r+s\le m,n$ and let $X_{rs}$, $E_{rs}$, $\mc G$, $\mc H$, $G$ and $\mc B$ be as in the final part of Section~\ref{ss.detcirc}. We note that if $m=n$, the map $(A,B)\mapsto(A,B^T)$ defines and isomorphism between $X_{rs}$ and the variety $X_{rs}$ from the first part of this section. The Levi $\mc K$ of $\mc P^-$ consists of the $(A,B)\in\mc P^-$ with the non-diagonal blocks zero. The parabolic $\mc Q$ consists of the $(A,B)\in\mc P$ with $A_{11}$ and $A_{33}$ lower triangular and with $B_{11}$ and $B_{33}$ upper triangular. Let $\mc T$ and $T$ be the maximal tori of $\mc G$ and $G$ that consist of pairs of diagonal matrices and let $B$ be the Borel subgroup of $G$ that consists of the pairs of matrices $(A,C)$ with $A$ upper and $C$ lower triangular. The standard coordinates on the left factor of $\mc T$ are denoted by $\ve_{11},\ldots,\ve_{m1}$, the standard coordinates on the right factor of $\mc T$ are denoted by $\ve_{12},\ldots,\ve_{n2}$ and the standard coordinates on $T$ are denoted by $\ve_1,\ldots,\ve_r,\delta_1,\ldots,\delta_s$. For $i\in\{1,\ldots,r-1\}$ we put $\alpha_i=\ve_i-\ve_{i+1}$ and for $j\in\{1,\ldots,s-1\}$ we put $\beta_j=-\delta_j+\delta_{j+1}$. We denote the corresponding $\mc B$-stable prime divisors by $D_i$ and $E_j$. So these are the $\mc B$-stable prime divisors that intersect $\mc P^-/\mc H$.
Each of the simple roots $\alpha_i$, $\beta_j$ of $G$ produces two simple roots of $\mc G$: $\alpha_{i1}=-\ve_{i1}+\ve_{i+1,1}$, $\alpha_{i2}=\ve_{i2}-\ve_{i+1,2}$, $\beta_{j1}=-\ve_{m-s+j,1}+\ve_{m-s+j+1,1}$ and $\beta_{j2}=\ve_{n-s+j,2}-\ve_{n-s+j+1,2}$.
The simple roots of $\mc G$ that are not roots of $\mc K$ are: $\gamma_{r1}=-\ve_{r1}+\ve_{r+1,1}$, $\gamma_{r2}=\ve_{r2}-\ve_{r+1,2}$, $\gamma_{s1}=-\ve_{m-s,1}+\ve_{m-s+1,1}$ and $\gamma_{s2}=\ve_{n-s,2}-\ve_{n-s+1,2}$. Here $\gamma_{r1}$ has to be omitted if $r=0$ or $r=m$, $\gamma_{r2}$ has to be omitted if $r=0$ or $r=n$ and similarly for the $\gamma_{sj}$. Furthermore, $\gamma_{r1}=\gamma_{s1}$ if $r+s=m$ and $\gamma_{r2}=\gamma_{s2}$ if $r+s=n$. We denote the $\mc B$-stable prime divisors corresponding to $\gamma_{rj}$ and $\gamma_{sj}$ by $D_{rj}$ and $D_{sj}$. Finally, we determine the ``weight map'' $X(T)\to X(\mc T)$ (i.e. the negative of composition with the orbit map). For $i\in\{1,\ldots,r\}$, $\ve_i$ is sent to $-\ve_{i1}+\ve_{i2}$ and for $j\in\{1,\ldots,s\}$, $\delta_j$ is sent to $\ve_{m-s+j,1}-\ve_{n-s+j,2}$.

Note that, if $m\le n$, the varieties $X_{m,0}$ and $X_{0,m}$ are open subsets of an affine space. In the proposition below these varieties and $X_{0,0}$ are excluded from consideration.
\begin{prop}\label{prop.classgroup}\
Assume $m\le n$.
\begin{enumerate}[{\rm(i)}]
\item Let $r\in\{1,\ldots,m-1\}$ and assume $m=n$. Then $X_{r,m-r}$ has two boundary divisors: $X_{r-1,m-r}$ and $X_{r,m-r-1}$. The class group is freely generated by $D_{r1}$ and $D_{r2}$. The canonical divisor is trivial.
\item Let $r\in\{1,\ldots,m-1\}$ and assume $m<n$. Then $X_{r,m-r}$ has no boundary divisors. The class group is freely generated by $D_{r1}$. A canonical divisor is given by $2(n-m)D_{r1}$.
\item Let $r,s\ge0$ with $(r,s)\ne(0,0)$ and $r+s<m$. Then $X_{rs}$ has no boundary divisors. The class group of $X_{rs}$ is freely generated by $D_{r1}$ if $s=0$, by $D_{s1}$ if $r=0$ and by $D_{r1}$ and $D_{s1}$ if $r,s>0$. A canonical divisor is given by $(n-m)(D_{r1}+D_{s1})$, where $D_{r1}$ has to be omitted if $r=0$ and $D_{s1}$ has to be omitted if $s=0$.
\end{enumerate}
\end{prop}

\begin{proof}
First we compute a canonical divisor. We have $\rho_{\mc Q}=\rho-\rho_{rs}$, where $\rho_{rs}$ is the half sum of the positive roots of $\mc G$ that are not roots of $R_u\mc Q$. Then $\rho_{rs}=\rho_{rs1}+\rho_{rs2}$, where $2\rho_{rs1}=-\sum_{r+1}^{m-s}(m+r-s+1-2i)\ve_{i1}$ and $2\rho_{rs2}=\sum_{r+1}^{n-s}(n+r-s+1-2i)\ve_{i2}$. It follows that $\la2\rho_{rs},\gamma_{r1}\ra=\la2\rho_{rs},\gamma_{s1}\ra=-(m-r-s-1)$ and $\la2\rho_{rs},\gamma_{r2}\ra=\la2\rho_{rs},\gamma_{s2}\ra=-(n-r-s-1)$. So, by Theorem~\ref{thm.canonicaldivisor}, a canonical divisor of $X_{rs}$ is given by
\begin{align*}
-\sum_iX_i-2\sum_iD_i-2\sum_jE_j&-(m-(r+s)+1)(D_{r1}+D_{s1})\\
&-(n-(r+s)+1)(D_{r2}+D_{s2})\,,
\end{align*}
where the $X_i$ denote the boundary divisors.

From the description of the stabiliser of $E_{rs}$ in the previous section, we deduce that $X_{rs}$ has dimension $(r+s)(m+n-(r+s))$. So there can only be boundary divisors if $m=n$ and $r+s=m$ (we know they must be amongst the $X_{r's'}$). Now assume that this is the case. Then we have two boundary divisors $X_{r-1,m-r}$ and $X_{r,m-r-1}$. Define $\lambda_r(t)$ and $\mu_r(t)$ as the $m\times m$ diagonal matrices with all diagonal entries equal to $1$ except for the $r$-th respectively $(r+1)$-th which is equal to $t$. Then the limits of $E_{r,m-r}$ along $(\lambda_r,0)$ and $(0,\mu_r)\in X(\mc T)$ are equal to $E_{r-1,m-r}$ and $E_{r,m-r-1}$. By Proposition~\ref{prop.limit} and Remark~\ref{rem.danger}, the images of $\lambda_r$ and $\mu_r$ in $\mb Q\ot Y(T)$ are $\ve_r^*$ and $\delta_1^*$, where we used the obvious notation for dual basis elements. Note that both these elements are anti-dominant for our choice of $B$. Now we can apply Proposition~\ref{prop.valuationcone} and Remark~\ref{rems.divisor}.1 and we obtain $(\ov{\ve_1})=D_1$, $(\ov{\ve_i})=D_i-D_{i-1}$, $i\in\{2,\ldots,r-1\}$, $(\ov{\ve_r})=X_{r-1,m-r}-D_{r-1}+D_{r1}+D_{r2}$, $(\ov{\delta_1})=X_{r,m-r-1}-E_1+D_{r1}+D_{r2}$, $(\ov{\delta_j})=-E_j+E_{j-1}$, $j\in\{2,\ldots,s-1\}$, $(\ov{\delta_s})=E_{s-1}$. Here we used that $D_{sj}=D_{rj}$, $j\in\{1,2\}$. So all $D_i$ and $E_j$ are $0$ in ${\rm Cl}(X_{r,m-r})$ and $D_{r1}+D_{r2}\equiv-X_{r-1,m-r}\equiv-X_{r,m-r-1}$. Assertion (i) follows. The relations for the other cases are obtained by omitting $X_{r-1,m-r}$ and $X_{r,m-r-1}$ and replacing the relation for $\ov{\delta_1}$ by $(\ov{\delta_1})=-E_1+D_{s1}+D_{s2}$. The proof of (ii) (here $D_{r1}=D_{s1}$) and (iii) can now be left to the reader.
\end{proof}

\begin{cornn}
In the cases (i)-(iii) of Proposition~\ref{prop.classgroup}, $X_{r,s}$ is Gorenstein if and only if $m=n$. In particular, all $G\times G$-orbit closures in the variety of non-invertible elements of the monoid $M$ are Gorenstein.
\end{cornn}

\noindent{\it Acknowledgement}. I would like to thank M.~Brion and L.~Renner for helpful email discussions and M.~Brion also for his comments on a first draft of this paper. This research was funded by a research grant from The Leverhulme Trust.

\bigskip

{\sc\noindent School of Mathematics,
Trinity College Dublin, College Green, Dublin 2, Ireland.
{\it E-mail address : }{\tt tanger@tcd.ie}
}

\end{document}